\documentclass[12pt]{article} 
\usepackage{latexsym,amsmath,amssymb}
\usepackage{graphicx}
\usepackage{youngtab}
\def\separation{\medskip}
\def\Z{Z\!\!\!Z}

\def\P{{\mathbb P}}
\def\C{{\mathbb C}}
\def\RR{{\mathbb R}}

\def\OO{{\mathcal O}}

\def\?{{\bf ??}}

\def\ker{{\rm ker}}

\newtheorem{alg}{Algorithm}
 \newtheorem{theorem}{Theorem}[section]
 
\newtheorem{prop}[theorem]{Proposition} 
\newtheorem{definition}[theorem]{Definition} 
\newtheorem{corollary}[theorem]{Corollary}
 
\newtheorem{remark}[theorem]{Remark}

\newcommand{\qed}{\hfill  $\Box$\separation}

\def\rig#1{\smash{ \mathop{\longrightarrow}\limits^{#1}}}

\numberwithin{equation}{section}                    
\begin{document}

\title{Effective methods for plane quartics,\\ their theta characteristics and  the Scorza map}
 \author{Giorgio Ottaviani} 
\date{} 
 \maketitle
 \begin{abstract} This is a revised version of the lecture notes prepared for the workshop on ``Plane quartics, Scorza map and related topics'', held in Catania,  January 19-21, 2016. The last section contains eight Macaulay2 scripts on theta characteristics and the Scorza map, with a tutorial. The first sections give an introduction to these scripts. The tutorial contains a list of the $36$ Scorza preimages of the Edge quartic.
\end{abstract}
\tableofcontents

\section{Introduction}
\subsection{How to write down plane quartics and their theta characteristics}
Plane quartics make a relevant family of algebraic curves because their
plane embedding is the canonical embedding. As a byproduct, intrinsic and projective geometry are strictly connected. A theta characteristic $\theta$
is by definition a square root  of the canonical bundle $K$, namely $2\theta=K$ in the additive notation, so $\theta$ is a priori an intrinsic object. There are $64=2^6$ theta characteristic, since the Jacobian variety of the plane quartic has real dimension $2g=6$. It is not a surprise that the  $64$ theta characteristics of a plane quartic show up in many projective constructions. There are $28$ odd theta characteristic (such that $h^0(\theta)=1$)
and $36$ even theta characteristic (such that $h^0(\theta)=0$).
The first well known fact is that the $28$ odd theta characteristics correspond to the $28$ bitangents of the plane quartic curve.
Indeed if a  line is tangent to the curve in two points $P$, $Q$, it is easy to see that $\theta:=P+Q$ satisfies $2\theta=K$, since the canonical divisor $K$ corresponds to the hyperplane divisor.
Moreover $h^0(P+Q)=1$ since $P+Q$ is an effective divisor. It is less known that the $36$ even theta characteristics may be visualized as the $36$ plane quartic curves that are the preimages through the Scorza map (see \S \ref{sec:scorza}). 
\begin{definition}\label{def:aro_system} A set of seven bitangents $\{\theta_1,\ldots, \theta_7\}$, that we identify with their odd theta characteristic,
is called an Aronhold system of bitangents if for every $\{i, j, k\}$ such that $1\le i< j< k\le 7$ we have $$h^0(2K-\theta_i-\theta_j-\theta_k)=0.$$
Replacing $2K=2\theta_i+2\theta_j$ this is equivalent to say that $\theta_i+\theta_j-\theta_k$ is an even theta characteristic.
For more informations see \cite[Defs. 4.1, 4.2]{Ser}.
\end{definition}

To write down explicitly a plane quartic and its theta characteristics, we will adopt the following four different descriptions
\begin{enumerate}
\item{$\bullet$} {\it A homogeneous polynomial $f$ of degree $4$ in $x_0, x_1, x_2$.} Computationally, this a vector with $15$ homogeneous coordinates.
\item{$\bullet$} {\it A symmetric linear determinantal representation of $f$, namely a symmetric $4\times 4$ matrix $A$ with linear entries in $x_0, x_1, x_2$, such that $\det(A)=f$.}
The exact sequence
$$0\rig{}\OO(-2)^4\rig{A}\OO(-1)^4\rig{}\theta\rig{}0$$
gives a even theta characteristic $\theta$, as a line bundle supported on the curve $\{f=0\}$. For a general $f$,
there are $36$ classes of matrices $A$ such that $\det(A)=f$ is given, up to $GL(4)$-congruence,
corresponding to the $36$ even theta characteristics. Computationally, we have a net of symmetric $4\times 4$ matrices $\langle A_0, A_1, A_2\rangle$.

\item{$\bullet$} Let $C=\{f=0\}$. {\it A symmetric $(3,3)$-correspondence in $C\times C$
  has the form $T_{\theta}=\{(x,y)\in C\times C | h^0(\theta +x-y)\neq 0\}$
corresponding to a even theta characteristic $\theta$.} Computationally, the divisor
$T_\theta$ is cut on the Segre variety $\P^2\times\P^2$ by six bilinear equations in $x, y$.
\item{$\bullet$} {\it Let $\{P_1,\ldots, P_7\}$ be seven general points in $\P^2$.
 The net of cubics through these seven points defines
a $2:1$ covering $\P^2(P_1,\ldots, P_7)\rig{\pi}\P^2$, with source the blow-up
of $\P^2$ at the points $P_i$, which ramifies over a
quartic $C\subset\P^2$.} In equivalent way, seven general lines in $\P^2$ define
a unique quartic curve $C$ such that the seven lines give a Aronhold system of bitangents of $C$.
 Computationally, these data can be encoded in a $7\times 3$ matrix (up to reorder the rows) or also in a $2\times 3$ matrix of the following type
\begin{equation}\label{eq:mat23}\begin{pmatrix}l_0&l_1&l_2\\q_0&q_1&q_2\end{pmatrix}\end{equation}
where $\deg l_i=1$, $\deg q_i=2$ and $l_i$ are independent. The maximal minors of this matrix vanish on seven points.
If $Z=\{P_1,\ldots, P_7\}$, then the resolution of the ideal $I_Z(3):=I_Z\otimes\OO(3)$ is
$$0\rig{}\OO(-1)\oplus\OO(-2)\rig{g}\OO^3\rig{}I_Z(3)\rig{}0$$
where the matrix of $g$ is the transpose of (\ref{eq:mat23}).
\end{enumerate}

The previous descriptions have an increasing amount of data, in the sense that
$2.$ and $3.$ are equivalent, while for the other items we have
\begin{equation}\label{eq:8-1}\{4. \textrm{seven points}\}\underbrace{\Longrightarrow}_{8:1} \{
\begin{array}{cc}2.&\textrm{symmetric determinant}\\
 3.&\textrm{symm. (3,3) corresp.}\end{array}\}\underbrace{\Longrightarrow}_{36:1} \{1. \textrm{quartic polynomial}\}\end{equation}

By forgetting the theta characteristic in the intermediate step, we get
the interesting correspondence
\begin{equation}\label{eq:288}\{ \textrm{seven points}\}\underbrace{\Longrightarrow}_{288:1}  \{ \textrm{quartic polynomial}\}\end{equation}

We will see how to ``move on the right'' in the diagrams
(\ref{eq:8-1}) and (\ref{eq:288}), from one description to another one.
These moves, connecting the several descriptions, are $SL(3)$-equivariant. Any family of quartics that is invariant by the $SL(3)$-action of projective linear transformations can be described by invariants or covariants in the above descriptions.

From the real point of view, it is interesting to recall the following Table from \cite[Prop. 5.1]{GH} for the theta characteristics of real quartics.

\begin{equation}\label{eq:tabletheta}\begin{array}{c|c|c}
\textrm{topological classification}&\textrm{\# real odd theta}&\textrm{\# real even theta}\\
\hline\\
\textrm{empty}&4&12\\
\textrm{one oval}&4&4\\
\textrm{two nested ovals}&4&12\\
\textrm{two non nested oval}&8&8\\
\textrm{three ovals}&16&16\\
\textrm{four ovals}&28&36\\
\end{array}\end{equation}

\subsection{Clebsch and L\"uroth quartics}\label{subsec:clelu}
A famous example is given by Clebsch quartics, which are the quartics which can be written as $f=\sum_{i=1}^5l_i^4$ (Waring decomposition), while the general quartic needs six summands consisting of $4$th powers of linear forms, in contrast with the naive numerical expectation.
Clebsch quartics can be detected
by a determinantal invariant of degree $6$ in the $15$ coefficients
of the quartic polynomial (description $1.$), called the catalecticant invariant or the Clebsch invariant.

Description of special quartics may be quite different depending on
the different descriptions we choose.
The archetypal example is that of L\"uroth quartics, which
by definition contain the $10$ vertices of a complete pentalateral, like in the picture
\begin{center}
\includegraphics[width=60mm]{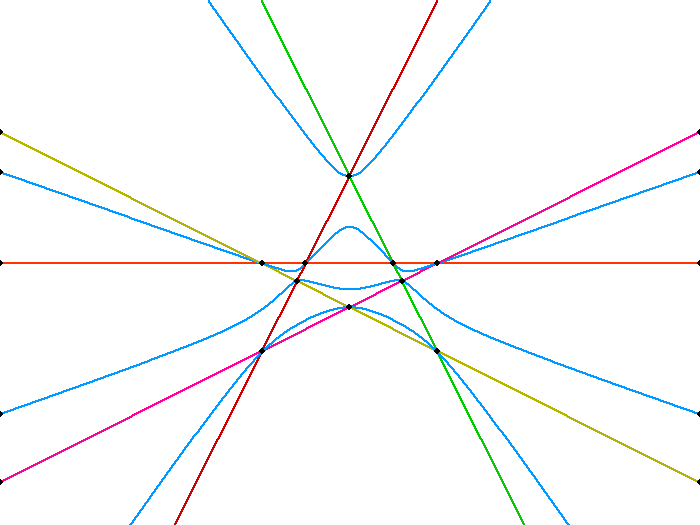}
\end{center}
The pentalateral defines in a natural way a particular even theta characteristic $\theta$, that it is called the pentalateral theta (the $10$ vertices of the pentalateral make the divisor $2K+\theta$).
The pair $(f,\theta)$ consisting of a L\"uroth quartic $f$
with a pentalateral theta $\theta$ can be described in a relatively easy way by 
looking at the determinantal representation (description $2.$): if the symmetric $4\times 4$ matrix
is $\sum_{i=0}^2x_iA_i$ then the condition of being L\"uroth with a pentalateral theta is expressed by the vanishing of the $6$th degree Pfaffian of the matrix
\begin{equation}\label{eq:paff6}\begin{pmatrix}0&A_2&-A_1\\
-A_2&0&A_0\\
A_1&-A_0&0\end{pmatrix},\end{equation}
see \cite[Theor. 4.1, Prop. 6.1]{Ot08}.
There is a second way to describe a L\"uroth quartic, if we can profit of the additional data of a Aronhold system of seven bitangents
(description $4.$). Indeed the seven points which give a L\"uroth quartic correspond
sursprisingly to the seven eigenvectors of a plane cubic.
See \S \ref{sec:eigentensors} for details, roughly this means that three exists a cubic polynomial $c$ such that in  (\ref{eq:mat23})
it holds $\frac{\partial c}{\partial l_i} = q_i$.
If we forget these additional data, it is tremendously difficult 
to detect a L\"uroth quartic looking just at it defining polynomial
(description $1.$).

L\"uroth quartic were studied deeply in the period 1860-1918, starting from
L\"uroth paper and their description in terms of net of quadrics, culminating with Morley brilliant description of the invariant in terms of the seven points,
showing finally that the degree of L\"uroth invariant, in terms of the fifteen coefficients of the quartic, is $54$.
The explicit expression of this degree $54$ invariant is challenging even nowadays with the help of a computer.  L\"uroth quartics
became again popular in 1977 when Barth showed that the jumping curve of a stable $2$-bundle on $\P^2$ with Chern classes $(c_1, c_2)=(0,4)$ is a L\"uroth quartic. LePotier and Tikhomirov showed in 2001 that the moduli space of the above bundles can be described in terms of L\"uroth quartics, the degree $54$ turned out to be a Donaldson invariant of $\P^2$.
\subsection{The Scorza map}
The Scorza map associates to a general quartic $f$ a pair $(S(f),\theta)$
where $S(f)$ is another quartic (the Aronhold covariant) and $\theta$ is a even theta characteristic on $S(f)$. Its precise definition needs the Aronhold invariant of plane cubics and it will be recalled in \S \ref{sec:scorza}. Its main property is the Theorem of Scorza that the map
$$f\mapsto (S(f),\theta)$$
 is dominant on the variety of pairs $(g,\theta)$ where $\theta$ is a even theta characteristic on the quartic $g$. In a previous paper, Scorza showed that if $f$ is Clebsch then $S(f)$ is L\"uroth with pentalateral $\theta$. This fact is the first step in the proof of the Theorem of Scorza. We will give a computational description of the Scorza map and its inverse, especially in Algorithms \ref{alg:sexticmodel} and \ref{alg:inversescorza}.

\subsection{Description of the content}
Sections from 2 to 6 describe a few basics about theta characteristics and Scorza map for plane quartics. The goal is to introduce the terminology to understand the algorithms, we refer to the literature for most of the proofs.  Some emphasis is given to Clebsch and L\"uroth quartics, which correspond through the Scorza map. More emphasis is given on the construction of a quartic from seven points, which give a Aronhold system of seven bitangents.
Section 7 is a brief survey about invariant theory of plane quartics.
Section 8 considers the link with the seven eigenvectors of a plane cubic, as sketched in \ref{subsec:clelu}. 
Section \ref{sec:alg} is the core of this paper and contains the Macaulay2\cite{GS} scripts.
These scripts are available as ancillary files of  the arXiv version of this paper
or by contacting the author. We have tried to use the verbatim text style to indicate
Macaulay2 input.

\subsection{Summary of the eight M2 scripts presented in \S \ref{sec:alg}}

\begin{enumerate}
\item{}
 INPUT: seven general lines $l_1,\ldots, l_7$

OUTPUT: the quartic having $l_1,\ldots, l_7$ as Aronhold system of bitangents

\item{} INPUT: a $2\times 3$ matrix with $2$-minors vanishing on $Z=\{l_1,\ldots, l_7\}$

 OUTPUT: the quartic having $Z$ as Aronhold system of bitangents

\item{} INPUT: seven general lines $l_1,\ldots, l_7$

 OUTPUT: a $8\times 8$ symmetric matrix (the bitangent matrix) collecting in each row the $8$ Aronhold systems of bitangents for the quartic having $l_1,\ldots, l_7$ as Aronhold system of bitangents, equivalent to $l_1,\ldots, l_7$. In particular, every principal $4\times 4$ minor of the bitangent matrix  gives a symmetric determinantal representation.

\item{} INPUT: a $2\times 3$ matrix with $2$-minors vanishing on $Z=\{l_1,\ldots, l_7\}$

 OUTPUT: a symmetric determinantal $4\times 4$ representation of the quartic having $Z$ as Aronhold system of bitangents

\item{} INPUT: a quartic $f$

OUTPUT: the image $S(f)$ through the Scorza map 

\item{} INPUT: a quartic $f$ and a point $q\in S(f)$

 OUTPUT: a determinantal representation of the image $(S(f),\theta)$ through the Scorza map 

\item{} INPUT: a determinantal representation of a quartic $g$ corresponding to $(g,\theta)$

 OUTPUT: the quartic $f$ such that the image $(S(f),\theta)$ through the Scorza map corresponds to $(g,\theta)$

{\it The tutorial contains a list of the $36$ Scorza preimages of the Edge quartic}

\item{} INPUT: a plane quartic $f$

 OUTPUT: the order of the automorphism group of linear invertible transformations which leave $f$ invariant
\end{enumerate}
The algorithms \ref{alg:edoardo}, \ref{alg:sexticmodel}, \ref{alg:inversescorza} are computationally expensive. I wonder if there
are simpler solutions and shortcuts, from the computational point of view.
\vskip 0.8cm

\subsection{Acknowledgements}
These notes were originally prepared for the Workshop on {\it “Plane quartics, Scorza map and related
topics”}, held in Catania, January 19-21, 2016. I warmly thank Francesco Russo for the idea and the choice of the topic and all participants for the stimulating atmosphere. Special thanks to Edoardo Sernesi and Francesco Zucconi for their very nice lectures \cite{Ser, Zuc} who gave the theoretical framework and allowed me to concentrate on the computational aspects. Algorithm \ref{alg:edoardo} was presented as an open problem in Catania, the idea for its solution, with the selection of two cubics and the two additional points were they vanish, is due to Edoardo Sernesi. I am deeply indebted to Edoardo and his insight for my understanding of plane quartics. Algorithm \ref{alg:inversescorza} arises from a question discussed with Bernd Sturmfels. These notes owe a lot to the computational point of view of the paper \cite{PSV} by D. Plaumann, B. Sturmfels and C. Vinzant. The topological classification of the $36$ Scorza preimages of the Edge quartic (see Algorithm \ref{alg:inversescorza}) was computed by Emanuele Ventura, after the workshop. Tha author is member of GNSAGA-INDAM.

\section{Apolarity, Waring decompositions}
Our base ring is $S^*(V)=\C[x_0,x_1,x_2]$. The dual ring of differential operators is
$S^*(V^\vee)=\C[\partial_0, \partial_1, \partial_2]$ with the action
satisfying $\partial_i(x_j)=\delta_{ij}$.

A differential operator $g\in S^*(V^\vee)$ such that $g\cdot f=0$
is called apolar to $f$.
Differential operators of degree $d$ can be identified with plane curves of degree $d$ in the dual plane. 

Apolarity is very well implemented in M2 by the command $\textrm{diff}$,
with the caveat that differential operators are written with the same variables
$x_i$ of the ring where they act.

\begin{definition}\label{def:polar}
Denote $P_a=\sum_i a_i\partial_i$. The polar of $f\in \mathrm{Sym}^dV$ at $a$ is $P_a(f)\in \mathrm{Sym}^dV$.
\end{definition}
If $f$ corresponds to the symmetric multilinear form
$f(x,\ldots, x)$ then $P_a(f)$ corresponds to the multilinear form
$f(a,x,\ldots, x)$.
It follows that after $d$ iterations we get
$$P_a^d(f)=\underbrace{P_a\circ\cdots\circ P_a}_{d\textrm{\ times}}(f) = d!f(a).$$

An example important for Morley construction is the following, which is discussed in \cite{OS1}.
Take a cubic $f$ with a nodal point $Q$.
Then $P_a(f)$ is the nodal conic consisting of the two nodal lines
making the tangent cone at $Q$.

Note that $f$ depends essentially on $\le 2$ variables (namely it is a cone)
if and only if there is a differential operator of degree $1$ apolar to $f$,
in this case we say that there is a line apolar to $f$.

The lines apolar to $f$ make the kernel of the contraction map
$$C^1_f\colon V^{\vee}\to\mathrm{Sym}^3V$$

In the same way, the conics apolar to $f$ make the kernel of the contraction map
$$C^2_f\colon\mathrm{Sym}^2V^{\vee}\to\mathrm{Sym}^2V$$
The map $C^2_f$ is called the middle catalecticant map and there are conics apolar to $f$
if and only if the middle catalecticant of $f$ vanishes.

In equivalent way, the Clebsch quartics $f$ of section \ref{subsec:clelu}
can be defined by the condition $\det C^2_f=0$.

\section{The Aronhold invariant of plane cubics}
References for this section: \cite{Ot,Stu,Do2,LO}.

The Aronhold invariant is the equation of the $SL(3)$-orbit of the Fermat cubic
$x^3+y^3+z^3$ in $\P^9=\P(\mathrm{ Sym}^3\C^3)$. By construction, it is an $SL(3)$-invariant in the $10$ coefficients of a plane cubic. In other words, it is the equation of the $3$-secant variety to
the $3$-Veronese embedding of $\P^2$, which is an hypersurface of degree $4$.

\begin{theorem}\cite[Ex. 1.2.1]{LO}, \cite[Theor. 1.2]{Ot}
Let $f(x,y,z)\in\mathrm{ Sym}^3\C^3$ be a homogeneous cubic polynomial. 
Let $C(f_x)$, $C(f_y)$, $C(f_z)$ be the three $3\times 3$ symmetric matrices
which are the Hessian of the three partial derivatives of $f$.

All the $8$-pfaffians of the $9\times 9$ skew-symmetric matrix
\begin{equation}\label{eq:aronhold}\begin{pmatrix}0&C(f_z)&-C(f_y)\\
-C(f_z)&0&C(f_x)\\
C(f_y)&-C(f_x)&0
\end{pmatrix}\end{equation}
coincide (up to scalar) with the Aronhold invariant. Note the beautiful analogy with (\ref{eq:paff6}).
\end{theorem}
Let $\mathrm{End}_0~\C^3$ be the space of traceless endomorphisms of $\C^3$.
The matrix (\ref{eq:aronhold}) describes\cite[\S 2]{Ot} the contraction

  $$A_{f}\colon\mathrm{End}_0~\C^3\to\mathrm{End}_0~\C^3,$$  which in the case $f=v^3$ satisfies

\begin{equation}\label{eq:eigen}A_{v^3}(M)(w)=\left(M(v)\wedge v\wedge w\right)v\quad\forall M\in \textrm{End}~\C^3,\quad\forall w\in \C^3.\end{equation}

\begin{theorem}[Nonabelian apolarity for plane cubics]\label{thm:nonab_cubics}
Let $f=l_1^3+l_2^3+l_3^3$  with $l_i$ not collinear linear forms.
Then $l_i$ are (symultaneous) eigenvectors of all the matrices $M\in\ker A_f\subset \textrm{End}_0~\C^3$.
This allows to recover $l_i$ from $f$.
\end{theorem}
\begin{proof} When $f=v^3$ the equation (\ref{eq:eigen}) shows that $Ker A_{v^3}=\{M\in End_0(\C^3)| v\textrm{\ is an eigenvector of\ }M\}$.
Hence we have the containment  $$Ker A_{f}\supseteq\{M\in End_0(\C^3)| l_i\textrm{\ are eigenvectors of\ }M\}$$ and equality holds since both spaces are $2$-dimensional.
\end{proof}\qed

\begin{remark} Let $M$ be a general matrix of $Ker A_{f}$. then $\ker A_f$ is obtained as the intersection of $\langle I, M, M^2\rangle$ with $\mathrm{End}_0~\C^3$.
In particular $l_i$ can be found as eigenvectors of $M$. An efficient implementation is via the numerical command \begin{verbatim}eigenvectors(sub(M,CC))\end{verbatim} of M2.
An alternative way to recover $l_i$ from $f$ is sketched in Proposition \ref{prop:hessian}. It is theoretically simpler but not so efficient from the computational point of view.
Moreover the technique of Theorem \ref{thm:nonab_cubics} can be generalized to other cases, for the case of general plane quintics see \cite[Algorithm 1]{OO}.

\end{remark}

The explicit expression of the Aronhold invariant $Ar$ (sometimes called also $S$ in the literature, that we cannot use to avoid ambiguity with the Scorza map) has $25$ monomials and it can be found in \cite{Stu} Prop. 4.4.7 or in 
\cite{DK} (5.13.1), or as output of the following M2 script

\begin{verbatim}
R=QQ[x,y,z,c_0..c_9]
x1=matrix{{x,y,z}}
x3=symmetricPower(3,x1)
f=(matrix{{c_0..c_9}}*transpose x3)_(0,0)
m=matrix{{0,z,-y},{-z,0,x},{y,-x,0}}
---following is 9*9 matrix
m9=diff(m,diff(x1,diff(transpose x1,f)))
---following is Aronhold invariant
aronhold=(mingens pfaffians(8,m9))_(0,0)
\end{verbatim}

By a slight abuse of notation we denote by $Ar$ also the corresponding multilinear form. The classical symbolic expression for $Ar$ is

$$Ar(x^3,y^3,z^3,w^3)=(x\wedge y\wedge z)(x\wedge y\wedge w)(x\wedge z\wedge w)(y\wedge z\wedge w).$$

\begin{prop}\label{prop:3orbits}\label{prop:closurefermat}
The closure of the orbit $SL(3)\cdot\left(x^3+y^3+z^3\right)$ contains the following three orbits
$$\begin{array}{c|c|c|c}
&\textrm{equation}&\textrm{dim}&\textrm{rk}\\
\hline\\
\textrm{Fermat}&x^3+y^3+z^3&8&3\\
\textrm{cuspidal}&y^2z-x^3&7&4\\
\textrm{smooth conic+tg line}&x(y^2-xz)&6&5
\end{array}$$
These are all the cubics with border rank three. Of course the closure contains also the cubics with border rank $\le 2$,
which make other three orbits.

The Aronhold invariant vanishes on the above three orbits and on all the cubics depending on
essentially one or two variables, they have border rank $\le 2$
and they are cones with a point as vertex.
\end{prop}

In the classical terminology, a Fermat cubic $l_1^3+l_2^3+l_3^3$
has a polar $3$-lateral given by $l_1l_2l_3$. A recipe to compute the
polar $3$-lateral in the first case of Prop. \ref{prop:closurefermat}, alternative to Theorem \ref{thm:nonab_cubics}
is given by the case (i) of the following Proposition.
\begin{prop}\label{prop:hessian}

(i) The Hessian of $l_1^3+l_2^3+l_3^3$ factors as $l_1l_2l_3$.

(ii) The Hessian of a cuspidal cubic splits as $l_1^2l_2$ where $l_1^2$
is the tangent cone at the cusp.

(iii) The Hessian of $\{\textrm{smooth conic}\}\cup\{\textrm{tg line\ } l\}$
is $l^3$.
\end{prop}

\begin{prop}{Real Fermat cubics}\cite{Ba}
The real cubics with complex border rank three make four $SL(3,\RR)$-orbits,
the Fermat case in Prop. \ref{prop:3orbits} splits into the two cases

(i') orbit of $x^3+y^3+z^3$, real Fermat

(i'') orbit if $(x+\sqrt{-1}y)^3+(x-\sqrt{-1}y)^3+z^3=2x^3-6x^2y+z^3$, imaginary Fermat
\end{prop}

\section{Three descriptions of an even theta characteristic}

\subsection{The symmetric determinantal description}\label{subsec:symmdet}
Let $A_0$, $A_1$, $A_2$ be three symmetric $4\times 4$ matrices.
The expression $$\det\left(xA_0+yA_1+zA_2\right)$$ defines a plane quartic
with a even theta characteristic $\theta$ given by
$$0\rig{}\OO_{\P^2}(-2)^4\rig{A}\OO_{\P^2}(-1)^4\rig{}\theta\rig{}0$$

As an example, the Edge quartic
\begin{equation}f=25(x^4+y^4+z^4)-34(x^2y^2+x^2z^2+y^2z^2)\label{eq:edge}\end{equation}
studied by Edge in \cite[\S 14]{Edge} (set $\gamma=-2$ )
has the symmetric determinantal representation found by Edge
$$\bgroup\begin{pmatrix}0&
     x+2 y&
     2 x+z&
     y-2 z\\
     x+2 y&
     0&
     y+2 z&
     -2 x+z\\
     2 x+z&
     y+2 z&
     0&
     x-2 y\\
     y-2 z&
     -2 x+z&
     x-2 y&
     0\\
     \end{pmatrix}\egroup$$


We refer to \cite{Ser,PSV} for the correspondence between the $28$ bitangents
of the quartic and the lines joining the eight base points of the net of quadrics. The seven lines joining one base point with the other seven base points make a Aronhold system of bitangents,
according to Def. \ref{def:aro_system}.
The seven corresponding points in $\P^3$ are Gale dual (see \cite{PSV,EP})
of the seven bitangents.

\subsection{The sextic model}
Let $K$ be the canonical bundle over a smooth plane quartic. The line bundle $K+\theta$ has $4$ independent sections which give a linear system of effective divisors of degree $6$.
It gives an embedding of the curve as a degree $6$ (and genus $3$) curve in $\P^3=\P\left(H^0(K+\theta)^\vee\right)$.
It has the resolution
\begin{equation}\label{eq:sextic}0\rig{}\OO_{\P^3}(-3)^3\rig{M}\OO_{\P^3}(-2)^4\rig{g}\OO_{\P^3}(1)\rig{}K+\theta\rig{}0\end{equation}
The symmetric determinantal description in \ref{subsec:symmdet} gives a $4\times 4\times 3$-tensor that we considered as a $4\times 4$ symmetric matrix with three linear entries.
The tensor has a second flattening as the $4\times 3$ matrix $M$
with four linear entries. The four  maximal minors of $M$ give the cubic equations of the sextic and define the map $g$
in (\ref{eq:sextic}).
The Edge quartic (\ref{eq:edge}) gives the following
$$M=\bgroup\begin{pmatrix}{u}_{1}+2 {u}_{2}&
      2 {u}_{1}+{u}_{3}&
      {u}_{2}-2 {u}_{3}\\
      {u}_{0}-2 {u}_{3}&
      2 {u}_{0}+{u}_{2}&
      2 {u}_{2}+{u}_{3}\\
      2 {u}_{0}+{u}_{3}&
      {u}_{1}-2 {u}_{3}&
      {u}_{0}+2 {u}_{1}\\
      -2 {u}_{1}+{u}_{2}&
      {u}_{0}-2 {u}_{2}&
      -2 {u}_{0}+{u}_{1}\\
      \end{pmatrix}\egroup$$

The isomorhism between the quartic model with coordinates $(x,y,z)$ and the sextic model with coordinates $(u_0,\ldots, u_3)$
is guaranteed by the system \begin{equation}\label{eq:bisystem}M(u)\cdot(x,y,z)^t=0.\end{equation} Given $u$ such that $\mathrm{rk} M(u)=2$,  the system 
(\ref{eq:bisystem}) defines a unique $(x,y,z)$.
Conversely, given $(x,y,z)$ on the quartic plane model, the system (\ref{eq:bisystem})
defines a unique $(u_0,\ldots, u_3)$ in the sextic space model.

The M2 code to construct $M$ from $A$ is the following

\begin{verbatim}
R=QQ[x,y,z,u_0..u_3]
A= matrix({{0, x + 2*y, z + 2*x, -2*z + y}, 
  {x + 2*y, 0, 2*z + y, z - 2*x}, 
	  {z + 2*x, 2*z + y, 0, x - 2*y}, 
	  {-2*z + y, z - 2*x, x - 2*y, 0}});
f=det(A)
uu=transpose matrix {{u_0..u_3}}
M=diff(x,A)*uu|diff(y,A)*uu|diff(z,A)*uu
\end{verbatim}

An elegant and alternative way to construct the divisor $K+\theta$ is the following

\begin{theorem}[Dixon]\label{thm:dixon}
Given a symmetric determinantal representation $A$, the four principal minors define
four cubics which are contact cubics, namely they cut the quartic $\det(A)$ in a nonreduced divisor
supported on a degree six divisor which is $K+\theta$.

More generally, if $A^{adj}$ is the adjugate matrix,
$u^tA^{adj}u$ parametrizes contact cubics, making a $3$-fold of degree $8$ in $\P^9$.
\end{theorem}

From the sextic model, one finds four independent sections $u_i\in H^0(K+\theta)$.
Then the sections $u_iu_j$ may be lifted to cubics in $\P^2$, with proper scaling guaranteed by the equation $(u_i+u_j)^2=u_i^2+u_j^2+2u_iu_j$.
Then the adjugate of the $4\times 4$ matrix $\left( u_iu_j\right)$
has degree $9$ polynomials which contain $f^2$ as a factor, after dividing by $f^2$
we get the corresponding symmetric determinantal description. 

\begin{remark} Dixon Theorem \ref{thm:dixon} generalizes to $d\times d$ symmetric representation,
in this case the degree $d(d-1)/2$ divisor, which is the support of the nonreduced divisor, is $H+\theta$.
\end{remark}

\begin{remark}\label{rem:symm} Note that a general $4\times 3$ matrix $M$ with linear entries in $u_0,\ldots, u_3$ defines a sextic curve of genus $3$,
but the line bundle which gives the embedding has the form $K+L$, with $L$ a degree $2$ line bundle which is not necessarily a theta characteristic.
The condition to be a theta characteristic is equivalent to the fact that the $4\times 4$ matrix with entries in $x_0,\ldots, x_2$ obtained by flattening
$M$ may be symmetrized by row/columns operations.\end{remark}

\subsection{The $(3,3)$-correspondence}\label{33corr}
References for this section: \cite{D,DK}.
Let $\theta$ be an even theta characteristic on $C$ of genus $3$.
It is defined a $(3,3)$ correspondence from the following divisor on $C\times C$
\begin{equation}\label{eq:ttheta}
T_\theta=\{(P,Q)\in C\times C| h^0(\theta+P-Q)>0\}.\end{equation} It follows from Serre duality that
the correspondence is symmetric. By Riemann-Roch, $\forall P\in C$, $\theta+P$ is linearly equivalent to a unique effective divisor
of degree $3$, hence we get a $(3,3)$ correspondence.

\begin{remark} An analogous $(g,g)$ correspondence may be defined starting from a general line bundle $L$ of degree $g-1$,
but this correspondence is in general not symmetric, the symmetry is guaranteed from $L$ being a theta characteristic (see Remark \ref{rem:symm}).
\end{remark}

In the sextic model the correspondence has the following form:
pick coordinates $(u_0,\ldots, u_3), (v_0,\ldots, v_3)$
and we have, for any $(u,v)\in C\times C$
$$(u,v)\in T_\theta\Longleftrightarrow v\cdot M(u)=0.$$

This works because $\ker M(u)$ is a $3$-secant line of the sextic model (see \cite{Sc}).


\section{The Aronhold covariant of plane quartics and the Scorza map.}\label{sec:scorza}
Main references for this section are \cite{DK,Ot08, OS1}. 
Recall by Definition \ref{def:polar} that $P_xf$ is the polar of $f$ at the point $x$.
The Scorza map is defined as $$f\mapsto (S(f),\theta),$$
where $S(f)=\{x\in \P^2|Ar(P_xf)=0\}$ is the Aronhold covariant (the notation with the letter $S$ goes back to Clebsch
and has nothing to do with Scorza,
see \S 7 in Ciani's monograph \cite{Cia})
and the $(3,3)$-correspondence $T_\theta$ which encodes $\theta$ as in (\ref{eq:ttheta}) is defined by
$$\{(x,y)\in\P^2\times\P^2|\textrm{rk}(P_xP_yf)\le 1\}.$$

The main result regarding the Scorza map is the following Theorem, proved by Scorza in 1899\cite{Sc}.

\begin{theorem}[Scorza]\label{thm:scorza}
The map $f\mapsto S(f)$ is dominant and is generically $36:1$.
\end{theorem}

Since there are $36$ even theta characteristic on a general quartic curve, Theorem \ref{thm:scorza} implies that the general pair $(C,\theta)$ where $C$ is a plane quartic and $\theta$ is an even theta characteristic on $C$ comes from a unique quartic $f$ through the Scorza map.

Assume that $P_xf=l_1^3+l_2^3+l_3^3$ and denote
$x_{ij}=\{l_i=l_j=0\}$. The divisor $\theta$ is linearly equivalent
to the divisor $$x_{12}+x_{13}+x_{23}-x.$$

This description allows to compute explicitly both the Scorza map and its inverse (see the Algorithm \ref{alg:inversescorza}). Note that the determinantal description can be obtained in the coefficient field of $f$ if $S(f)$ contains a point $x$ lying in the same field.

\begin{theorem}\label{thm:clebschlueroth}\cite[\S 7.3]{DK}
If $f$ is Clebsch then $(S(f),\theta)$ is L\"uroth with the pentalateral $\theta$. Conversely the general $(g,\theta)$, where $g$ is a L\"uroth curve with pentalateral $\theta$,
comes from a unique Clebsch curve $f$ such that $S(f)=g$.

\end{theorem}

\begin{corollary} The infinitely many decomposition of a Clebsch quartic,
with $f=\sum_{i=0}^4l_i^4$ and $l_i$ circumscribed to the conic $C$,
give infinitely many pentalateral inscribed in $S(f)$.
\end{corollary}

A reflection on this Theorem allows to understand why the invariant description of L\"uroth condition for the pair $(f,\theta)$ (regarding net of quadrics, see (\ref{eq:paff6})) is much simpler than
the one for $f$ itself.

Since the Scorza map is $SL(3)$-equivariant, we remark the following consequence
\begin{prop} Let $Aut(f)$ be the automorphsim group of linear transformation of $\P^2$ leaving the quartic $f$ invariant.
Then
$$Aut(f)\subseteq Aut(S(f))$$
is a group inclusion. In particular the order of $Aut(f)$ divides the order of $Aut(S(f))$.
\end{prop}

As a corollary, note that both double conics (having $Aut(f)=SL(2)$) and Klein quartic (having $Aut(f)$ the simple group of order $168$, the group of higher order among all irreducible quartics) both satisfy
$Aut(f)=Aut(S(f))$.

Moreover, the $36$-preimages $S^{-1}S(f)$divide into $SL(3)$-orbits $O_i$ for $i=1,\ldots, k$, which are the same as $Aut(S(f))$.
These orbits are studied in the literature regarding the action of the automorphism group on the even spin structures of the curve.

For any $f\in O_i$ the size of $Aut(f)$ is fixed and we get $$|O_i|=\frac{|Aut(S(f))|}{|Aut(f)|},$$
which can be used jointly with the obvious identity

$$\sum_{i=1}^k |O_i|=36.$$

See the comments and the tables before the M2 script of Algorithm \ref{alg:inversescorza} in \S \ref{sec:alg}.
\section{Contact cubics and contact triangles}

For any even theta characteristics $\theta$, the effective divisors of degree six corresponding
to $K+\theta$ can be computed with the following trick.
Let $A$ be a $4\times 4$ symmetric determinantal representation corresponding to $\theta$, as in \ref{subsec:symmdet}.
Any principal minor of $A$  defines a contact cubic, which cuts the quartic
in a nonreduced divisor, supported on $K+\theta$. More generally,
for any $u=(u_0,\ldots, u_3)$ there is a contact cubic  given by
\begin{equation}\label{eq:contactcub}\det\left(\begin{array}{cc}A&u^t\\
u&0\end{array}\right).\end{equation}
This was the classical formula which gives the entries of the adjugate matrix of $A$,  it is quite convenient from the computational point of view
when regarding matrices with symbolic entries, compare with Dixon Theorem \ref{thm:dixon}.

When the diagonal elements of $A$ are zero we have a further description, indeed the contact cubics are the triangles given by three bitangents, that can be found from the matrix

$$\begin{pmatrix}0&l_{01}&l_{02}&l_{03}\\
l_{01}&0&l_{12}&l_{13}\\
l_{02}&l_{12}&0&l_{23}\\
l_{03}&l_{13}&l_{23}&0\end{pmatrix}$$
where $l_{ij}$ are bitangents. The notation $l_{ij}$ for the $28$ bitangents come from the the $28$ pairs joining
the eight base points $P_0,\ldots, P_7$ of a net of quadrics in $\P^3$, as in \S \ref{subsec:symmdet} and it is implemented in Algorithm 3 in \S 9.
In this case the four principal minors are contact triangles,
like $$\det\begin{pmatrix}0&l_{01}&l_{02}\\
l_{01}&0&l_{12}\\
l_{02}&l_{12}&0\end{pmatrix}$$
which gives $l_{01}l_{02}l_{12}$.

 There are $56$ contact cubics given by three bitangents in the family (\ref{eq:contactcub}) of contact cubics,
all together they are $56\cdot 36= 2016$. All of these triples $\{\theta_i, \theta_j, \theta_k\}$  have the six contact points which do not lie on a conics,
which is equivalent to $h^0(2K-\theta_i- \theta_j- \theta_k)=0$, exactly as in Definition \ref{def:aro_system}.
The $36$ families of contact cubics correspond to $36$ $2$-Veronese $3$folds in $\P^9$ (of degree $8$).
They do not meet the $2$-secant variety of the  $3$-Veronese surface, namely the variety of triangles given by three collinear lines. They meet the variety of all triangles in $120$ points.
There are $8$ strictly biscribed triangles, according to Mukai\cite{Muk},
each one counts with multiplicity $8$
for a total of $64$, indeed note that $56+64=120=8\cdot 15$
(degree of intersection of Veronese $3$fold with the variety of triangles).

There are other $28$ families of contact cubics, beyond the $36$ families as in (\ref{eq:contactcub}),  that can be constructed starting from any bitangent $\ell$ in the following way.
The contact cubics $c$ corresponding to $\ell$ satisfy the equation $f=\ell c+q^2$ for some conic $q$. Each of these families is $3$-dimensional.  In each of these families there
are $45$ contact triangles given by three bitangents. 
Note that $45\cdot 28=1260$ and $1260+2016=3276={{28}\choose 3}$
so that we have counted exactly once all triples of bitangents. The $1260$ triples are special since they have six contact points on a conic. They meet 
the $3$-Veronese surface of triple lines in one point of multiplicity
$6$, which corresponds to the bitangent at power $3$, like $l^3$.
In any case, the expression $ml^3+q^2=f$ does not hold for any line $m$ and conic $q$.
Label the $28$ bitangents as $ij$ where $0\le i<j\le 7$, corresponding to the pairs of base points of the net of quadrics,
as in \ref{subsec:symmdet}.
These $45$ contact triangles divide in two types, $30$ of them like $12.23.34$,
other $15$ of them like $12.34.56$ (see \cite[\S 13]{Cia}).
In the first type there is a fourth bitangent $41$ such that the $8$ contact points lie on a conic,
there are $210$ $4$ples of this kind. Also in the second type there is a fourth bitangent $07$ such that the $8$ contact points lie on a conic,
there are $105$ $4$ples of this kind. Altogether, there are $210+105=315$ $4$ples of bitangents such that the $8$ contact points lie on a conic.

\section{The invariant ring of plane quartics}
 It is worth to advertise that the complete determination of invariant ring of plane quartics is a big achievement of computer algebra, the final step was presented at MEGA 2013 in Frankfurt by Andreas-Stephan Elsenhans \cite{Els}, relying on previous work by Shioda and Dixmier, so solving a classical question
which went back to Emmy Noether's doctoral thesis.

The following result was conjectured by Shioda in 1967, who computed the Hilbert series.
\begin{theorem}
The invariant ring $\displaystyle\left[\oplus_dSym^d\left(Sym^4\C^3\right)\right]^{SL(3)}$
is generated by invariants of degree
 $3, 6, 9, 9, 12, 12, 15, 15, 18, 18, 21, 21, 27$. The relations are known.
\end{theorem}
Dixmier found in 1987 the invariants of degree $3, 6, 9, 12, 15, 18, 27$ which are algebraically independent, so that
the invariant ring is an algebraic extension of the ring generated by these ones (primary invariants).
The invariants up to degree $18$ (more the discriminant of degree $27$) can be found in Salmon book\cite{Sal}, compare also with
\cite[\S 7]{Cia}.
So only the invariants of degree $21$ were missing in the 19th century.
Apparently the first who produced the complete generators of invariant ring was T. Ohno in an unpublished work in 2007.
Note the elementary fact that the degree of any invariants is divisible by $3$. The Clebsch invariant,
defining quartics of rank $5$, has degree $6$.
Some classical facts regarding the cubic invariant are recalled in \cite{OttWM}.

\section{The link with the seven eigentensors of a plane cubic}\label{sec:eigentensors}
\def\niente{Roberts Theorem, proved first in 1889 (see \cite{OS1}), states that a general pair $(q,f)$, where $q$ ia a plane conic and $f$ is a plane cubic,
has a unique (symultaneous) Waring decomposition given by four lines $l_i$ for $i=1,\ldots, 4$ and scalars $\lambda_i$ such that
$$\left\{\begin{array}{ccc}q&=&\sum_{i=1}^4l_i^2\\
f&=&\sum_{i=1}^4\lambda_il_i^3\end{array}\right.$$}


Let $q$ be a nondegenerate quadratic form on $V$. Given a plane cubic $f\in\mathrm{Sym}^3V^\vee$, an eigentensor of $f$ is $v\in V$ such that 
\begin{equation}\label{eq:eigencub}f(v,v,x)=\lambda q(v,x)\end{equation} for every $x\in V$. When $q$ is the euclidean metric (in the real setting),
which gives an identification between $V$ and $V^\vee$, the previous equation
can be written as $f(v^2)=\lambda v$, which is the way the eigentensor equation is commonly written in the numerical setting, and it is the natural generalization of the 
eigenvector condition for symmetric matrices.

In the metric setting, the generalization is more transparent.
Indeed, the eigenvectors of a symmetric matrix $q$ are the {\it critical points of the distance function} from $q$ to the Veronese variety $(\P^2,\OO(2))$.
In the same way, the seven eigentensors are the seven {\it critical points of the distance function} from $f$ to the Veronese variety $(\P^2,\OO(3))$
, see \cite{EDpaper}.

A dimensional count shows that seven eigenvectors of a cubic cannot be seven general point, so that it is interesting to understand their special position according to (\ref{eq:288}). The following is the geometric counterpart of \cite[Prop. 5.3]{ASS}.

\begin{theorem}[Bateman]

(i) The seven points $p_1,\ldots, p_7$ which are eigentensors of a cubic 
satisfy the following property:

the seven nodal conics $C_i$ for $i=1,\ldots 7$ which correspond to
$P_{P_i}H_i$, where $H_i$ is the unique cubic passing through all $p_j$
and singular at $p_i$ are harmonic, that is $\Delta C_i=0$,
where $\Delta$ is the Laplacian.

(ii) Seven points $v_i$ are eigentensors of a cubic $f$ with respect to some nondegenerate conic $q$ as in (\ref{eq:eigencub})
if and only the seven points give a L\"uroth quartic in the correspondence
(\ref{eq:288}).
 \end{theorem}

\begin{proof} Part (i) is a reformulation of Morley differential identity,
see formula (20) in \S 9 of \cite{OS1}, where $\theta=x_0^2+x_1^2+x_2^2$
and the fact that (with the notations of \cite{OS1}), the Morley form
$M(P_i,X)$ coincides with $H_i(X)$, see the paragraph after Corollary 3.3 in \cite{OS1}. Part (ii) is a reformulation of Theorem 10.4 in \cite{OS1}
\end{proof}

\section{Eight algorithms and Macaulay2 scripts, with a tutorial}\label{sec:alg}

\begin{alg}\label{alg:1}\

\begin{itemize}
\item{} INPUT: seven general lines $l_1,\ldots, l_7$
\item{} OUTPUT: the quartic having $l_1,\ldots, l_7$ as Aronhold system of bitangents
\end{itemize}
\end{alg}
The steps of the algorithm are the following
\begin{enumerate}
\item{} Compute the $2\times 3$ matrix
$$\begin{pmatrix}a_0(x)&a_1(x)&a_2(x)\\q_0(x)&q_1(x)&q_2(x)\end{pmatrix}$$
with $\deg a_i=1$, $\deg q_i=2$ degenerating on $Z=\{l_1,\ldots, l_7\}$ seen as seven points in the dual space.
This matrix is computed from the resolution of the ideal vanishing on the points and it can be taken as input (as in Algorithm \ref{alg:2}).
\item{} Construct the net of cubics passing through $Z$ as
$$\mathrm{net}(x,y)=\det\begin{pmatrix}a_0(y)&a_1(y)&a_2(y)\\a_0(x)&a_1(x)&a_2(x)\\q_0(x)&q_1(x)&q_2(x)\end{pmatrix}$$
\item{} Construct the jacobian of the net as the determinant of the $3\times 3$ matrix
$$\mathrm{jac}(x)=\det\left[\frac{\partial^2\mathrm{net}}{\partial x_i\partial y_j}\right]$$
which is a sextic (in coordinates $x$)  nodal at the seven points.
\item{} The quartic in output is obtained by eliminating $x$ from the four equations $\mathrm{jac}(x)$ and $y_i-\frac{\partial\mathrm{net}(x,y)}{\partial y_i}$.
\end{enumerate}

As a running example,
start with input given the seven points which are the rows of the following matrix
$$\bgroup\begin{pmatrix}1&
      2&
      0\\
      2&
      0&
      1\\
      0&
      1&
      {-2}\\
      5&
      5&
      3\\
      5&
      {-3}&
      5\\
      3&
      5&
      {-5}\\
      {-1}&
      1&
      1\\
      \end{pmatrix}\egroup$$

get as output 
the quartic
$$f=25 {y}_{0}^{4}-34 {y}_{0}^{2} {y}_{1}^{2}+25 {y}_{1}^{4}-34 {y}_{0}^{2}
      {y}_{2}^{2}-34 {y}_{1}^{2} {y}_{2}^{2}+25 {y}_{2}^{4}.$$

In alternative, start from
$$\bgroup\begin{pmatrix}{x}_{0}&
       {x}_{1}&
       {x}_{2}\\
       (3 {x}_{0}^{2}+2 {x}_{0} {x}_{1}+{x}_{2}^{2})&
       ({x}_{0}^{2}+2 {x}_{1} {x}_{2})&
       ({x}_{1}^{2}+2 {x}_{0} {x}_{2})\\
       \end{pmatrix}\egroup$$

where the second row is the gradient of
$x_0^2x_1+x_1^2x_2+x_2^2x_0+x_0^3$
and get as output the L\"uroth quartic
$$f=2 {y}_{0}^{3} {y}_{1}-2 {y}_{0}^{2} {y}_{1}^{2}-{y}_{0}
       {y}_{1}^{3}-{y}_{0}^{3} {y}_{2}-2 {y}_{0}^{2} {y}_{1} {y}_{2}+4 {y}_{0}
       {y}_{1}^{2} {y}_{2}+2 {y}_{1}^{3} {y}_{2}+{y}_{0}^{2} {y}_{2}^{2}-2
       {y}_{0} {y}_{1} {y}_{2}^{2}+{y}_{1}^{2} {y}_{2}^{2}-{y}_{0}
       {y}_{2}^{3}-{y}_{1} {y}_{2}^{3}.$$

The M2 script is the following
\begin{verbatim}
KK=QQ
R=KK[x_0..x_2,y_0..y_2]
x1=matrix{{x_0..x_2}}
y1=matrix{{y_0..y_2}}
----the rows of p2 contain the seven points, below are shown two samples
p2=random(R^{7:0},R^{3:0})
p2=matrix{{1,2,0},{2,0,1},{0,1,-2},{5,5,3},{5,-3,5},{3,5,-5},{-1,1,1}}
---computation of the 2*3 matrix mat
I7=minors(2,p2^{0}||x1)
for i from 0 to 6 do I7=intersect(I7,minors(2,p2^{i}||x1))
r7=res I7
betti r7
mat=transpose r7.dd_2
------in alternative one can start from a 2*3 matrix as above 
---we construct now the Morley form
nc=(sub((mat)^{0},apply(3,i->(x_i=>y_i)))||mat)
net2=diff(matrix{{y_0..y_2}},det(nc))
---jacobian of the net, is a plane sextic nodal at seven points
jac=ideal(det(diff(transpose x1,net2)))
-----get the quartic f eliminating x_i from the net of cubics and the jacobian 
ff=eliminate({x_0,x_1,x_2},ideal(matrix{{y_0..y_2}}-net2)+jac)
f=((gens ff)_(0,0))
----f is our OUTPUT
---check the seven starting points are really bitangents
for i from 0 to 6 do print(i, degree ideal(f,(p2^{i}*transpose y1)_(0,0)),
 degree radical ideal(f,(p2^{i}*transpose y1)_(0,0)))

\end{verbatim}

\begin{alg}\label{alg:2}\

\begin{itemize}
\item{} INPUT: a $2\times 3$ matrix with $2$-minors vanishing on $Z=\{l_1,\ldots, l_7\}$
\item{} OUTPUT: the quartic having $Z$ as Aronhold system of bitangents
\end{itemize}
\end{alg}

This is just the combination of steps $2\ldots 4$ of Algoritm \ref{alg:1}.

A typical application is the matrix
$\bgroup\begin{pmatrix}{x}_{1}&
      {x}_{2}&
      {x}_{0}\\
      {x}_{0}^{2}&
      {x}_{1}^{2}&
      {x}_{2}^{2}\\
      \end{pmatrix}\egroup$
which gives the Klein quartic
${y}_{0} {y}_{1}^{3}+{y}_{1} {y}_{2}^{3}+{y}_{0}^{3} {y}_{2}$.
Let $\tau=e^{2\pi\sqrt{-1}/7}$ a  $7$-th root of unity.
In this case, the seven lines in $Z$ are represented by the columns of the matrix

$$\begin{pmatrix}y_0&y_1&y_2\end{pmatrix}\cdot\bgroup\begin{pmatrix}1&
      \tau^{3}&
      \tau^6&
      \tau^{2}&
      \tau^{5}&
      \tau&
      \tau^{4}\\
      1&
      \tau^{5}&
      \tau^{3}&
      \tau&
      \tau^6&
      \tau^{4}&
      \tau^{2}\\
      1&
      \tau^{2}&
      \tau^{4}&
      \tau^6&
      \tau&
      \tau^{3}&
      \tau^{5}\\
      \end{pmatrix}\egroup$$
but the script does not need $\tau$ and works on the field where the $2\times 3$ matrix
is defined.

The M2 script is the following

\begin{verbatim}
KK=QQ
R=KK[x_0..x_2,y_0..y_2]
x1=matrix{{x_0..x_2}}
y1=matrix{{y_0..y_2}}
-----with the following matrix we find Klein quartic !!
mat=matrix{{x_1,x_2,x_0},{x_0^2,x_1^2,x_2^2}}
nc=(sub((mat)^{0},apply(3,i->(x_i=>y_i)))||mat)
net2=diff(matrix{{y_0..y_2}},det(nc))
codim ideal net2, degree ideal net2
---jacobian of the net, is a plane sextic nodal at seven points
jac=ideal(det(diff(transpose x1,net2)))
-----get the quartic f eliminating x_i from the net of cubics and the jacobian 
ff=eliminate({x_0,x_1,x_2},ideal(matrix{{y_0..y_2}}-net2)+jac)
f=((gens ff)_(0,0))
---f is our output
\end{verbatim}
\begin{alg}\

\begin{itemize}
\item{} INPUT: seven general lines $l_1,\ldots, l_7$
\item{} OUTPUT: a $8\times 8$ symmetric matrix (the bitangent matrix) collecting in each row the $8$ Aronhold systems of bitangents for the quartic having $l_1,\ldots, l_7$ as Aronhold system of bitangents, equivalent to $l_1,\ldots, l_7$. In particular, every principal $4\times 4$ minor of the bitangent matrix  gives a symmetric determinantal representation.
\end{itemize}
\end{alg}

These are the steps
\begin{enumerate}
\item{} Write the coordinates of the seven points $\{l_1,\ldots, l_7\}$ as columns of a $3\times 7$ matrix $P$ and compute a Gale transform in $\P^3$ given by the seven rows $\{m_1,\ldots m_7\}$ of a $7\times 4$ matrix $M$ such that $PM=0$.
These are well defined modulo $SL(4)$-action.
\item{} Compute the net of quadrics through $\{m_1,\ldots m_7\}$, spanned by the $4\times 4$ symmetric matrices $Q_0$, $Q_1$, $Q_2$.
\item{} Compute the symmetric determinantal representation of the quartic curve as $\det(\sum_i x_iQ_i)$.
\item{} Compute (with a saturation), the eighth base point $m_0$ of the net and correspondingly stack it as a first row over $M$ so obtaining the $8\times 4$ matrix $M_0$.
The theory guarantees that the $28$ lines $m_im_j$ for $0\le i<j\le 7$ are the $28$ bitangent of the quartic constructed in the previous step, but we have
to apply a linear projective transformation as in next step.
\item{} Compute the linear projective transformation $g$ from the space of lines in $\P^3$ through $m_0$ to our dual $\P^2$ which takes the line $m_0m_i$ to $l_i$ for $i=1,\ldots, 7$.
\item{} The bitangent matrix is $M_0(\sum_i g(x)_iQ_i)M_0^t$ (see \cite[(3.4)]{PSV}).
\end{enumerate}

An example is given by the following seven lines in input
 $\bgroup\begin{pmatrix}{x}_{0}\\
      {x}_{1}\\
      {x}_{2}\\
      {x}_{0}+{x}_{1}+{x}_{2}\\
      {x}_{0}+2 {x}_{1}+3 {x}_{2}\\
      2 {x}_{0}+3 {x}_{1}+{x}_{2}\\
      3 {x}_{0}+{x}_{1}+2 {x}_{2}\\
      \end{pmatrix}\egroup$
The corresponding quartic is

{\footnotesize $81x_0^4+198x_0^3x_1-41x_0^2x_1^2-198x_0x_1^3+81x_1^4-198x_0^3x
      _2-1561x_0^2x_1x_2-1561x_0x_1^2x_2+198x_1^3x_2-41x_0^2x_2^2-
      1561x_0x_1x_2^2-41x_1^2x_2^2+198x_0x_2^3-198x_1x_2^3+81x_2^4$}

and the bitangent matrix in output has the following first three columns (the whole matrix is too big to be printed, it can be found by running the M2 script).

{\tiny $$\bgroup\begin{pmatrix}0&
      3 {x}_{0}&
      3 {x}_{1}&
      \\
      3 {x}_{0}&
      0&
      (777/143) {x}_{0}+(1239/143) {x}_{1}-(378/143) {x}_{2}\\
      3 {x}_{1}&
      (777/143) {x}_{0}+(1239/143) {x}_{1}-(378/143) {x}_{2}&
      0\\
      3 {x}_{2}&
      (1239/143) {x}_{0}-(378/143) {x}_{1}+(777/143) {x}_{2}&
      -(378/143) {x}_{0}+(777/143) {x}_{1}+(1239/143) {x}_{2}\\
      11 {x}_{0}+11 {x}_{1}+11 {x}_{2}&
      (126/13) {x}_{0}+(287/13) {x}_{1}+(133/13) {x}_{2}&
      (133/13) {x}_{0}+(126/13) {x}_{1}+(287/13) {x}_{2}\\
      -{x}_{0}-2 {x}_{1}-3 {x}_{2}&
      -(756/143) {x}_{0}-(448/143) {x}_{1}-(525/143) {x}_{2}&
      (119/143) {x}_{0}-(189/143) {x}_{1}-(1113/143) {x}_{2}\\
      -2 {x}_{0}-3 {x}_{1}-{x}_{2}&
      -(189/143) {x}_{0}-(1113/143) {x}_{1}+(119/143) {x}_{2}&
      -(392/143) {x}_{0}-(84/143) {x}_{1}-(161/143) {x}_{2}\\
      -3 {x}_{0}-{x}_{1}-2 {x}_{2}&
      -(84/143) {x}_{0}-(161/143) {x}_{1}-(392/143) {x}_{2}&
      -(525/143) {x}_{0}-(756/143) {x}_{1}-(448/143) {x}_{2}\\
      \end{pmatrix}\egroup$$
}

The M2 script is the following

\begin{verbatim}

restart
R=QQ[x_0..x_2,y_0..y_3]
x1=matrix{{x_0..x_2}}
y2=symmetricPower(2,matrix{{y_0..y_3}})
y1=matrix{{y_0..y_3}}
p2=matrix{{1,2,0},{2,0,1},{0,1,-2},{5,5,3},{5,-3,5},{3,5,-5},{-1,1,1}}
--- p2 contains seven points in P2
------p3 contains the seven points in P^3 which are the Gale transform 
p3=gens kernel transpose p2
---p3 contains the seven points in P3 which are Gale dual, defined up to PGL(4)
qp3=symmetricPower(2,p3^{0})
for i from 1 to 6 do qp3=qp3||symmetricPower(2,p3^{i})
net3=y2*gens kernel qp3
---net3 is the net of quadrics in P3
---note that even if the net is defined up to PGL(4), 
----the locus of singular quadrics is well defined in the net
----the important fact is that the net of quadrics is dual to the original P^2
for i from 0 to 2 do Q_i=diff(transpose y1,diff( y1,net3_(0,i)))
--for i from 0 to 2 do print Q_i
f=det(sum(3,i->x_i*diff(transpose y1,diff( y1,net3_(0,i)))))
---f is the plane quartic curve, living in projective plane of nets of quadrics, 
---it is not yet in the right coordinate system
----now we use the eighth point to make a projectivity with starting plane
pp3=minors(2,p3^{0}||y1)
for i from 1 to 6 do pp3=intersect(pp3,minors(2,p3^{i}||y1))
codim pp3, degree pp3
---pp3 is the ideal of seven points in P^3, obtained by Gale tranform
eighth=saturate(ideal(net3),pp3)
cord8=sub(matrix{apply(4,i->(y_i%eighth))},y_3=>1)
---cord8 contains coordinates of 8th point.
cc=sub(net3,apply(4,i->(y_i=>p3_(0,i)+2*cord8_(0,i))))
for i from 1 to 6 do cc=cc||sub(net3,apply(4,j->(y_j=>p3_(i,j)+2*cord8_(0,j))))
----cc contains the coordinates of the seven lines through the 8th point, 
---in the coordinates of the net
---now we look for projectivity between cc and p2, in the SAME order of points
R1=QQ[g_0..g_8]
gg=transpose genericMatrix(R1,3,3)
-----following are set of points we want to make in projection
-----we impose conditions to a unknown projective transformation gg
p21=sub(p2,R1),cc1=sub(cc,R1)
IG=minors(2,p21^{0}*gg||cc1^{0})
for i from 1 to 3 do IG=IG+minors(2,p21^{i}*gg||cc1^{i})
codim IG, degree IG
for i from 4 to 6 do IG=IG+minors(2,p21^{i}*gg||cc1^{i})
codim IG, degree IG
gc=sub(matrix{apply(9,i->(g_i)%IG)},g_8=>1)
gm=submatrix(gc,{0..2})||submatrix(gc,{3..5})||submatrix(gc,{6..8})
----gm is the projective transformation found
det gm
gi=inverse (sub(gm,QQ))
ff=sub(f,apply(3,i->(x_i=>(gi*transpose x1)_(i,0))))
----ff is the quartic in the starting(correct)  coordinate system
---qq is the determinantal representation of the quartic
-- with given Aronhold system of bitangents
qq=sum(3,i->((gi*transpose x1)_(i,0))*Q_i)
factor(det(qq))
---we want to find now the bitangent matrix, according to Hesse
---np3 cotains the eight base points of the net
np3=cord8||p3
btm=np3*qq*(transpose np3)*(1/20)
---btm is the bitangent matrix

---check that all 28 lines are bitangents
for i from 0 to 6 do for j from i+1 to 7 do print(i,j,
degree radical ideal(ff,btm_(i,j)))
\end{verbatim}

\begin{alg}\ \label{alg:edoardo}

\begin{itemize}
\item{} INPUT: a $2\times 3$ matrix with $2$-minors vanishing on $Z=\{l_1,\ldots, l_7\}$
\item{} OUTPUT: a symmetric determinantal $4\times 4$ representation of the quartic having $Z$ as Aronhold system of bitangents
\end{itemize}
\end{alg}

Starting with the
the matrix
$\bgroup\begin{pmatrix}{x}_{1}&
      {x}_{2}&
      {x}_{0}\\
      {x}_{0}^{2}&
      {x}_{1}^{2}&
      {x}_{2}^{2}\\
      \end{pmatrix}\egroup$
which gives the Klein quartic
${y}_{0} {y}_{1}^{3}+{y}_{1} {y}_{2}^{3}+{y}_{0}^{3} {y}_{2}$
we find the following determinantal representation
$$\bgroup\begin{pmatrix}2 {x}_{2}&
      0&
      0&
      {-{x}_{1}}\\
      0&
      2 {x}_{0}&
      0&
      {-{x}_{2}}\\
      0&
      0&
      2 {x}_{1}&
      {-{x}_{0}}\\
      {-{x}_{1}}&
      {-{x}_{2}}&
      {-{x}_{0}}&
      0\\
      \end{pmatrix}\egroup$$

These are the steps of the algorithm
\begin{enumerate}
\item{} Compute the net of cubics through the seven points and choose two independent cubics in the net. In the matrix description,
these are given by two $2\times 2$ minors.
\item{} With a saturation, compute the two additional points where the two cubics vanish and compute
the $4$-dimensional space of conics $\langle Q_0(x),\ldots Q_3(x)\rangle$ vanishing at these two points.
\item{} A Gale tranform in $\P^3$ can be found by eliminating $x_i$ from the net of cubics and from the four equations $y_i-Q_i(x)$ for $i=0,\ldots 3$.
The result is an ideal $p3$ in coordinates $y$ generated by three quadrics and one cubic, vanishing in seven points.
With these conventions, the net $\mathrm{net3}$ of the three quadrics $\langle F_0(y),\ldots, F_2(y)\rangle$
vanish on a eighth point which has coordinates
$(y_0,y_1,y_2,y_3)=(0,0,0,1)$.
\item{} We eliminate $y_i$ from the ideal generated by $x_i-F_i(y)$ and from the three cubics(in $y$)
obtained by eliminating $y_3$ from $p3$. The result is a net of cubics (in $x_i$) that we call $\mathrm{I7net}$.
\item{} We find a linear $3\times 3$ transformation in coordinates $x_i$
from the net of plane cubics we started with, to the net $\mathrm{I7net}$, after this linear transformation is applied to $F_i(y)$, the resulting net of quadrics gives the determinantal representation in output.
\end{enumerate}

\begin{verbatim}
KK=QQ
R=KK[x_0..x_2]
x1=matrix{{x_0..x_2}}
mat=matrix{{x_1,x_2,x_0},{x_0^2,x_1^2,x_2^2}}
----we redefine I7 accordingly, so that I7 comes from mat
I7=gens minors(2,mat)
---chose two cubics in the minors of mat and find two additional points.
twocub=ideal submatrix(gens minors(2,mat),,{0,1})
quad=super basis(2,saturate(twocub,ideal I7)) ---four conics in x_i
R1=KK[x_0..x_2,y_0..y_3,g_0..g_8]
x1=matrix{{x_0..x_2}}
x3=symmetricPower(3,matrix{{x_0..x_2}})
y1=matrix{{y_0..y_3}}
---we find now p3, gale transform
p3=eliminate({x_0,x_1,x_2},sub(ideal I7,R1)+ideal(y1-sub(quad,R1)))
codim p3, degree p3
betti p3
----p3 is seven points in P^3, three quadrics and one cubic in y_i
----net3, only the three quadrics in y_i remain
net3=submatrix(gens p3,,{0..2})
codim ideal net3, degree ideal net3
i7net=eliminate({y_0,y_1,y_2,y_3},ideal(matrix{{x_0,x_1,x_2}}-net3)+
eliminate({y_3},p3))
codim i7net, degree i7net
dnet3=diff(symmetricPower(3,matrix{{x_0..x_2}}), transpose gens i7net)
gg=transpose genericMatrix(R1,g_0,3,3)
----we find now a 3*3 linear transformation from I7 to I7net
----we stack to I7, considered as 3*10 matrix, each of the cubics in I7net
--- (as a 10-dimensional vector), transformed with a unknown 3*3 matrix gg
---and we ask that the resulting 4*10 matrix is degenerate
-----if all the 4*4 minors are too many, one may select only some of them
IG=ideal(0_R1)
for j from 0 to 2 do IG=IG+minors(4,diff(x3,sub(sub((I7)_(0,j),R1),
apply(3,i->(x_i=>((gg)*transpose x1)_(i,0)))))||dnet3)
time SIG=saturate (IG,ideal(det(gg)));
codim SIG, degree SIG
----with Edge quartic we find three solutions, 
--the following decompose command can be skipped if SIG vanishes on just 1 point
dgg=decompose SIG
gc=sub(matrix{apply(9,i->(g_i)%dgg_0)},g_8=>1)
gm=submatrix(gc,{0..2})||submatrix(gc,{3..5})||submatrix(gc,{6..8})
----gm is the projective transformation found
nq3=net3*(transpose gm)*transpose x1
----following mrep is the determinantal representation found
mrep=diff(transpose y1,diff( y1,nq3))
\end{verbatim}
\begin{alg}\

\begin{itemize}
\item{} INPUT: a quartic $f$
\item{} OUTPUT: the image $S(f)$ through the Scorza map 
\end{itemize}
\end{alg}

The script uses the expression of the Aronhold invariant  in \cite{LO}, example 1.2.1. Let $H(f_x)$ be the Hessian matrix of the cubic $f_x$.
All the $8$-Pfaffians of the following $9\times 9$ matrix (written in $3\times 3$ block form) coincide, up to scalar, with the equation of $S(f)$.

$$\left[\begin{array}{ccc}
0&H(f_z)&-H(f_y)\\
-H(f_z)&0&H(f_x)\\
H(f_y)&-H(f_x)&0\end{array}\right]$$

In the M2 script we check that starting from $f=x^4+y^4+z^4+2\alpha(x^3y+y^3z+z^3x)-\alpha(x^2y^2+y^2z^2+x^2z^2)+(\alpha-\frac{2}{3})(xy^3+yz^3+zx^3)-4(x^2yz+xy^2z+xyz^2)$, with
$\alpha=\frac{-1+\sqrt{-7}}{2}$,
then $S(f)=xy^3+yz^3+zx^3$.

{\footnotesize
\begin{verbatim}
K2=toField ((ZZ/31991)[t]/ideal(t^2+7))
R=K2[y_0..y_2]
xx=basis(1,R)
a=(-1+t)/2
---change f in following line
f=y_0^4+y_1^4+y_2^4+2*a*(y_0^3*y_1+y_1^3*y_2+y_2^3*y_0)-a*(y_0^2*y_1^2+y_1^2*y_2^2+y_0^2*y_2^2)+
(a-2/3)*(y_0*y_1^3+y_1*y_2^3+y_2*y_0^3)-4*(y_0^2*y_1*y_2+y_0*y_1^2*y_2+y_0*y_1*y_2^2)
m=matrix{{0,y_2,-y_1},{-y_2,0,y_0},{y_1,-y_0,0}}
m9= diff(m,diff(xx,diff(transpose xx,f)));
scorzamap=(mingens pfaffians(8,m9))_(0,0)
---output is y_0*y_1^3+y_0^3*y_2+y_1*y_2^3
\end{verbatim}
}

\begin{alg}\ \label{alg:sexticmodel}

\begin{itemize}
\item{} INPUT: a quartic $f$ and a point $q\in S(f)$
\item{} OUTPUT: a determinantal representation of the image $(S(f),\theta)$ through the Scorza map 
\end{itemize}
\end{alg}

The steps are the following
\begin{enumerate}
\item{}  Compute the polar $P_q(f)$, which is a Fermat cubic curve since by assumption
$q\in S(f)$. The Hessian determinant of $P_q(f)$ was classically called the Polihessian and it splits in three lines, its singular locus gives the three points in $S(f)$ corresponding to $q$ in the $(3,3)$-correspondence
defined by the $\theta $.  Add to the divisor given by these three points a generic hyperplane divisor,
call $DP$ the resulting degree $7$ divisor. We want to find all the effective divisor of degree $6$ linearly equivalent to $DP-q$. To achieve this goal call $SE=R/(S(f))$ the quotient ring, call $IP/(S(f))$ and $I_q/(S(f))$
the images respectively of the ideals of $DP$ and $q$ in the quotient ring , pick the first generator $h$
of $IP/(S(f))$  and compute
$ h\cdot I_q/(S(f))\colon IP/(S(f))$. This is an ideal generated by four cubics that represent the four generators of $H^0(C,K+\theta)$, in the following sense. The common base locus of $C_i$ is a degree six effective divisor on $S(f)$, say $K+\theta$. Each cubic of the system vanishes on $K+\theta$ and in additional six points which
define a degree six effective divisor linearly equivalent.  Call $C_i$ , for $i=0,\ldots, 3$, the representatives of these cubics in the original ring $R$. 
Introduce new variables $u_0,\ldots, u_3$ and eliminate $(x,y,z)$ from the ideal generated by
$S(f)$ and $u_i-C_i(x,y,z)$. We get the ideal of the sextic curve in $\P^3$ which is 
defined by the linear system $K+\theta$.

\item{} The resolution of the ideal of the sextic curve in $\P^3$ found at previous step gives a $4\times 3$ $C$ matrix with linear entries. Note that a different flattening gives a $4\times 4$ matrix with linear entries,
but we have to find a symmetric representative of it. 

\item{} The correct relation between coordinates $u_i$ and $(x,y,z)$
is given by the relation $C\cdot\begin{pmatrix}x\\y\\z\end{pmatrix}=0$ (see \ref{33corr}).
After elimination of $u_i$ we find that the polynomial $u_iu_j$ are cubic polynomials in $(x,y,z)$.
We construct the matrix $md$ with entries $u_iu_j$ which are cubic polynomials in $x,y,z$
Main problem is to correctly scale them
To scale we find separately the cubic polynomial corresponding 
to $(u_i+u_j)^2$ (recorded in matrix $ad$) and we use the identity $(u_i+u_j)^2=u_i^2+u_j^2+2u_iu_j$

\item{} We conclude last step with a symmetric $4\times 4$ matrix $pmd$ with $(i,j)$ entry
given by the cubic polynomial corresponding to $u_iu_j$. Now we use Dixon technique, as explained
in \cite{PSV}, we adjugate $pmd$ and we factor out $(S(f))^2$ from each entry. The output is the wished
linear determinantal representation.

\end{enumerate}

A nice example is to start from

$f=x^{4}+y^{4}+z^{4}+(-6 \sqrt{2}-6)\left( x^{2} y^{2}+x^{2} z^{2}+y^{2} z^{2}\right)$

which satisfies $S(f)=x^{4}+ y^{4}+ z^{4}$.

Choosing the point $q=(1,0,e^{(\pi /4)\sqrt{-1}})\in S(f)$ we get the following symmetric determinantal representation of the Fermat quartic
(set $\alpha=\sqrt{-2}+\sqrt{-1}$ and $\beta=\sqrt{-2}-\sqrt{-1}$)

{\tiny
$$\bgroup\begin{pmatrix}(-\alpha-1) x+(\beta-2 \sqrt{2}+5) z&
     (\alpha+1) y&
     (\beta-1) x+(-\beta-1) z&
     (-\beta+1) y\\
     (\alpha+1) y&
     (\beta-1) x+(2 \sqrt{-2}-\sqrt{-1}-\sqrt{2}+3) z&
     (-\beta+1) y&
     (\alpha+1) x+(-\alpha+3) z\\
     (\beta-1) x+(-\beta-1) z&
     (-\beta+1) y&
     (\alpha+1) x+(\alpha-1) z&
     (-\alpha-1) y\\
     (-\beta+1) y&
     (\alpha+1) x+(-\alpha+3) z&
     (-\alpha-1) y&
     (-\beta+1) x+(\sqrt{-1}-\sqrt{2}+1) z\\
     \end{pmatrix}\egroup$$
}
Actually the determinant is $-32 \sqrt{-1}\left( x^{4}+ y^{4}+ z^{4}\right)$.
This result leads to great admiration for Edge who, without a computer, finds in \cite[\S 12]{Edge} other determinantal representations
of the Fermat quartic curve,  with essentially the same technique but, due to computational tricks available only in the Fermat example,
his output turns out to be simpler and more elegant !

A second example is the quartic $f$ with Waring decomposition given by 
$f=x^4+y^4+(x+y)^4+(x+y+z)^4+(x+2y+3z)^4+(-5x+7y-11z)^4$.
Note that the point $(0,0,1)$ belongs to the first three summands, so that the polar cubic $f_z$ is Fermat and
$(0,0,1)\in S(f)$.
The output found with the algorithm is messy but can be found in less than one minute on a PC.
We do not print it here.

The M2 script is the following

\begin{verbatim}

------------------INPUT: f, and a point q in S(f) 
------------------OUTPUT: a determinantal representation of S(f)
K2=QQ
--K2=toField(QQ[t]/ideal(t^2-2))
R=K2[x,y,z]
x2=basis(2,R)
x1=basis(1,R)
f=x^4+y^4+(x+y)^4+(x+y+z)^4+(x+2*y+3*z)^4+(-5*x+7*y-11*z)^4---the point q=(0,0,1)
--- kills first three summands, so it belongs to S(f)
-----
hes=det diff(transpose x1,diff(x1,diff(z,f)))---it is the polohessian triangle
singhes=ideal(diff(x1,hes))
---computation of S(f), called scf
m=matrix{{0,z,-y},{-z,0,x},{y,-x,0}}
scf= (mingens(pfaffians(8,diff(x1,diff(transpose x1,diff(m,f))))))_(0,0)
sub(scf,{x=>0,y=>0,z=>1})
----we have the theta given by singhes-PP
---now we add the hyperplane divisor and we find a degree six 
---effective divisor linearly equivalent

SE=R/(scf)
DN= ideal(x,y)--negative part of divisor (1 point)
DP= intersect(sub(singhes,SE),ideal(random(1,SE)))--positive part 
---of divisor (3+4=7 points)
h=(gens DP)_(0,0)
LD=((h*DN) : DP);---four cubics which have base points  
---on the effective divisor and give the embedding of scf in P^3
betti LD
---from the four cubics vanishing on the six points on scf 
---we get the sextic embedding in P^3
----by computing linear relation between (u,x) with Groebner basis
R2=K2[x,y,z,u_0..u_3]
ld=sub(LD,R2)

J=ideal(sub(scf,R2),u_0-(gens ld)_(0,0),u_1-(gens ld)_(0,1),
u_2-(gens ld)_(0,2),u_3-(gens ld)_(0,3))
----we offer two alternatives to compute the 4*3 cs2 matrix
---the first one, commented below is more direct
--rJ=res eliminate({x,y,z},J)
--cs2=(rJ.dd)_2;
---the second alternative turns out to be computationally cheaper
JJ=gens gb(J);
betti JJ
cs1=diff(transpose matrix{{x,y,z}},submatrix(JJ,,{0..3}));
cs2=transpose cs1;

-----cs2 is the 4*3 matrix

----the following commands reconstruct the 4*4 determinantal representation
----from the 4*3 matrix, which should be "symmetrizable"
------we want first to construct the matrix pmd with entries u_i*u_j
--- which are cubic polynomials in x,y,z
------main problem is to correctly scale them
------to scale we use the identity (u_i+u_j)^2=u_i^2+u_j^2+2*u_i*u_j

----md is the matrix pmd, up to scale, this elimination procedure is crucial
--- and works only for "symmetrizable" 4*3 matrix
md=mutableMatrix(R2,4,4)
--next loop runs in some seconds
for i from 0 to 3 do for j from i to 3 do 
md_(i,j)=(gens eliminate({u_0,u_1,u_2,u_3},saturate(ideal(cs2*
transpose matrix{{x,y,z}} )+ideal(u_i*u_j),ideal(u_0,u_1,u_2,u_3))))_(0,0)
for i from 0 to 3 do for j from 0 to i-1 do md_(i,j)=md_(j,i)
-----ad contains the auxiliary polynomials (u_i+u_j)^2
ad=mutableMatrix(R2,4,4)
for i from 0 to 3 do for j from i to 3 do 
ad_(i,j)=(gens eliminate({u_0,u_1,u_2,u_3},saturate(ideal(
cs2*transpose matrix{{x,y,z}} )+ideal((u_i+u_j)^2),
ideal(u_0,u_1,u_2,u_3))))_(0,0)
-----S2 is the ring with the coefficients p_(i,j), q_(i,j) needed to scale
S2=K2[x,y,z,u_0..u_3,p_(0,0)..p_(3,3),q_(0,0)..q_(3,3)]
md=sub(matrix md,S2)
ad=sub(matrix ad,S2)
x3=symmetricPower(3,matrix{{x,y,z}})

-----I contains the conditions needed in order to scale
I=ideal()
for i from 0 to 2 do for j from i+1 to 3 do
I=I+ideal(diff(x3,p_(i,i)*md_(i,i)+p_(j,j)*md_(j,j)+
2*p_(i,j)*md_(i,j)-q_(i,j)*ad_(i,j)))
codim I, degree I
------pp contains the solution to the system given by I, 
---this works if q_(2,3) is different from zero,
---otherwise the solution to I has to be extracted in alternative way
pp=mutableMatrix(S2,4,4)
for i from 0 to 3 do for j from i to 3 do pp_(i,j)=(sub(p_(i,j)%I,q_(2,3)=>1))
pp
------pmd is our goal matrix
pmd=mutableMatrix(S2,4,4)
for i from 0 to 3 do for j from i to 3 do 
pmd_(i,j)=pp_(i,j)*md_(i,j)
for i from 0 to 3 do for j from 0 to i-1 do pmd_(i,j)=pmd_(j,i)
pmd=matrix(pmd)

------now we use Dixon technique, we adjugate pmd 
---and we factor out (scf)^2 from each entry
mdf=mutableMatrix(S2,4,4)
for i from 0 to 3 do for j from 0 to 3 do mdf_(i,j)=(-1)^(i+j)*
(quotientRemainder(det submatrix'(pmd,{i},{j}),sub((scf)^2,S2)))_0
---mdf is our output
mdf=matrix mdf
factor det mdf
\end{verbatim}

\begin{alg}\ \label{alg:inversescorza}

\begin{itemize}
\item{} INPUT: a determinantal representation of a quartic $g$ corresponding to $(g,\theta)$
\item{} OUTPUT: the quartic $f$ such that the image $(S(f),\theta)$ through the Scorza map corresponds to $(g,\theta)$
\end{itemize}
\end{alg}

The steps are the following
\begin{enumerate}
\item{} We find the sextic model of the quartic in $\P^3$ by transformimg the $4\times 4\times 3$ representation in a $4\times 3$ matrix $M(u)$ with linear entries in $u=(u_0,\ldots, u_3)$. The maximal minors of $M$ define the sextic model given by the linear system $K+\theta$.
\item{} We define the $(3,3)$ correspondence introducing dual coordinates $v=(v_0,\ldots, v_3)$
and the ideal $\mathrm{cor33}$ generated by the maximal minors of $M(u)$, $M(v)$ and by the bilinear equations given by $u\cdot M(v)$ and $v\cdot M(u)$.
\item{} We lift the $(3,3)$-correspondence to the quartic model with dual coordinates $a=(a_0,\ldots, a_2)$ and $b=(b_0,\ldots, b_2)$. This works by eliminating $u$ and $v$ from the saturation over
$u$ and $v$ of the ideal generated by $\mathrm{cor33}$ and by the bilinear equations $M(u)\cdot a^t$ and
$M(v)\cdot b^t$. We get $6$ biquadratic equations in $(a, b)$ that we call $\mathrm{eli6}$
 This step is computationally expensive. At the end we convert the six equations in a $6\times 36 $ matrix, where the $36$ columns correspond to the $36$ biquadratic monomials in $(a,b)$, called $\mathrm{maab}$.
\item{} Let $f_c$ be a generic quartic given by $15$ entries $c_i$, we compute the $6$ biquadratic equations 
in $(a, b)$ obtained by the condition $\mathrm{rk}(P_aP_bf_c)\le 1$. We convert these equations
 in a $6\times 36$ matrix, called $\mathrm{mab}$, with entries depending quadratically on $c_i$.
\item{} We impose that the rows of two matrices  $\mathrm{maab}$ and $\mathrm{mab}$ span the same $6$-dimensional space. This gives quadratical equations in $c_i$. {\it Note that Scorza map has entries which are quartic in $c_i$, so that this ``tour de force'' has an effective gain at the end.}
In principle this can be obtained by stacking each row of  $\mathrm{mab}$ to  $\mathrm{maab}$ and computing all the (maximal) $7\times 7$-minors. This is computationally too expensive, so we choose only some of these minors. For general quartics, the minors which contain the first $6$ columns suffice. In the M2 script
it is sketched an alternative choice of minors which works for the Klein quartic.
\item{} The ideal of  the quadratical equations in $c_i$ obtained in previous step has to be saturated with base locus of Scorza map. The output is a unique point in $c_i$ which gives our output quartic.
If the output is bigger, we have to repeat the previous step by adding more minors.
\end{enumerate}

As a running example, we may start from the Edge quartic in (\ref{eq:edge}) with the symmetric determinantal representation already given
 and find as output
\begin{equation}\label{eq:4ovals}x^4+y^4+z^4+30(x^2y^2+x^2z^2+y^2z^2).\end{equation}
The algorithm can be iterated for other $35$ times, as described in \cite{PSV}.
More precisely we first find the bitangent matrix of (\ref{eq:edge}) with Algorithm 3,
then we apply \cite[Theorem 3.9]{PSV} with all $\{0,i,j,k\}$
such that $\{i,j,k\}\subset\{1..7\}$.
All $36$ quartics are real, in agreement with (\ref{eq:tabletheta}).

\begin{enumerate}
\item{} mother case, $x^4+30x^2y^2+y^4+30x^2z^2+30y^2z^2+z^4$
\item{}$(0, 1, 2, 3), x^4+(10/3)x^2y^2+y^4+(10/3)x^2z^2+(10/3)y^2z^2+z^4)$
\item{}$(0, 1, 2, 4), x^4-(84/125)x^3y+(54/25)x^2y^2+(492/125)xy^3+(81/625)y^4+(28/25)x^3z-(36/5)x^2yz-(108/25)xy^2z-(492/125)y^3z+6x^2z^2+(36/5)xyz^2+(54/25)y^2z^2+(28/25)xz^3+(84/125)yz^3+z^4)$
\item{}$(0, 1, 2, 5), (625/81)x^4+(700/81)x^3y+(1250/27)x^2y^2+(700/81)xy^3+(625/81)y^4+(140/27)x^3z+(500/9)x^2yz-(500/9)xy^2z-(140/27)y^3z+(50/3)x^2z^2-(100/3)xyz^2+(50/3)y^2z^2-(820/27)xz^3+(820/27)yz^3+z^4)$
\item{}$(0, 1, 2, 6), x^4+(164/9)x^3y+6x^2y^2+(164/9)xy^3+y^4+(164/9)x^3z+12x^2yz+12xy^2z+(164/9)y^3z+6x^2z^2+12xyz^2+6y^2z^2+(164/9)xz^3+(164/9)yz^3+z^4)$
\item{}$(0, 1, 2, 7), (81/625)x^4-(492/125)x^3y+(54/25)x^2y^2+(84/125)xy^3+y^4+(492/125)x^3z-(108/25)x^2yz+(36/5)xy^2z+(28/25)y^3z+(54/25)x^2z^2-(36/5)xyz^2+6y^2z^2-(84/125)xz^3+(28/25)yz^3+z^4)$
\item{}$(0, 1, 3, 4), (81/625)x^4-(492/125)x^3y+(54/25)x^2y^2+(84/125)xy^3+y^4-(492/125)x^3z+(108/25)x^2yz-(36/5)xy^2z-(28/25)y^3z+(54/25)x^2z^2-(36/5)xyz^2+6y^2z^2+(84/125)xz^3-(28/25)yz^3+z^4)$
\item{}$(0, 1, 3, 5), x^4+(164/9)x^3y+6x^2y^2+(164/9)xy^3+y^4-(164/9)x^3z-12x^2yz-12xy^2z-(164/9)y^3z+6x^2z^2+12xyz^2+6y^2z^2-(164/9)xz^3-(164/9)yz^3+z^4)$
\item{}$(0, 1, 3, 6), (625/81)x^4+(700/81)x^3y+(1250/27)x^2y^2+(700/81)xy^3+(625/81)y^4-(140/27)x^3z-(500/9)x^2yz+(500/9)xy^2z+(140/27)y^3z+(50/3)x^2z^2-(100/3)xyz^2+(50/3)y^2z^2+(820/27)xz^3-(820/27)yz^3+z^4)$
\item{}$(0, 1, 3, 7), x^4-(84/125)x^3y+(54/25)x^2y^2+(492/125)xy^3+(81/625)y^4-(28/25)x^3z+(36/5)x^2yz+(108/25)xy^2z+(492/125)y^3z+6x^2z^2+(36/5)xyz^2+(54/25)y^2z^2-(28/25)xz^3-(84/125)yz^3+z^4)$
\item{}$(0, 1, 4, 5), x^4-(15/16)y^4+(17/4)x^3z+(45/4)xy^2z+(3/2)x^2z^2+(17/4)xz^3+z^4)$
\item{}$(0, 1, 4, 6), 15x^4+y^4+45x^2yz-17y^3z+24y^2z^2-17yz^3+z^4)$
\item{}$(0, 1, 4, 7), (9/26)x^4+(15/13)x^2y^2+(9/26)y^4-(15/13)x^2z^2-(15/13)y^2z^2+z^4)$
\item{}$(0, 1, 5, 6), -(1/16)x^4-(15/8)x^2y^2-(1/16)y^4+(15/8)x^2z^2+(15/8)y^2z^2+z^4)$
\item{}$(0, 1, 5, 7), 15x^4+y^4-45x^2yz+17y^3z+24y^2z^2+17yz^3+z^4)$
\item{}$(0, 1, 6, 7), x^4-(15/16)y^4-(17/4)x^3z-(45/4)xy^2z+(3/2)x^2z^2-(17/4)xz^3+z^4)$
\item{}$(0, 2, 3, 4), x^4-(164/9)x^3y+6x^2y^2-(164/9)xy^3+y^4+(164/9)x^3z-12x^2yz+12xy^2z-(164/9)y^3z+6x^2z^2-12xyz^2+6y^2z^2+(164/9)xz^3-(164/9)yz^3+z^4)$
\item{}$(0, 2, 3, 5), (81/625)x^4+(492/125)x^3y+(54/25)x^2y^2-(84/125)xy^3+y^4+(492/125)x^3z+(108/25)x^2yz+(36/5)xy^2z-(28/25)y^3z+(54/25)x^2z^2+(36/5)xyz^2+6y^2z^2-(84/125)xz^3-(28/25)yz^3+z^4)$
\item{}$(0, 2, 3, 6), x^4+(84/125)x^3y+(54/25)x^2y^2-(492/125)xy^3+(81/625)y^4+(28/25)x^3z+(36/5)x^2yz-(108/25)xy^2z+(492/125)y^3z+6x^2z^2-(36/5)xyz^2+(54/25)y^2z^2+(28/25)xz^3-(84/125)yz^3+z^4)$
\item{}$(0, 2, 3, 7), (625/81)x^4-(700/81)x^3y+(1250/27)x^2y^2-(700/81)xy^3+(625/81)y^4+(140/27)x^3z-(500/9)x^2yz-(500/9)xy^2z+(140/27)y^3z+(50/3)x^2z^2+(100/3)xyz^2+(50/3)y^2z^2-(820/27)xz^3-(820/27)yz^3+z^4)$
\item{}$(0, 2, 4, 5), (1/15)x^4+(17/15)x^3y+(8/5)x^2y^2+(17/15)xy^3+(1/15)y^4-3xyz^2+z^4)$
\item{}$(0, 2, 4, 6), x^4-30x^2y^2-16y^4+30x^2z^2-30y^2z^2+z^4)$
\item{}$(0, 2, 4, 7), -(15/16)x^4+y^4+(45/4)x^2yz+(17/4)y^3z+(3/2)y^2z^2+(17/4)yz^3+z^4)$
\item{}$(0, 2, 5, 6), -(15/16)x^4+y^4-(45/4)x^2yz-(17/4)y^3z+(3/2)y^2z^2-(17/4)yz^3+z^4)$
\item{}$(0, 2, 5, 7), x^4-(10/3)x^2y^2+(26/9)y^4+(10/3)x^2z^2-(10/3)y^2z^2+z^4)$
\item{}$(0, 2, 6, 7), (1/15)x^4-(17/15)x^3y+(8/5)x^2y^2-(17/15)xy^3+(1/15)y^4+3xyz^2+z^4)$
\item{}$(0, 3, 4, 5), -16x^4-30x^2y^2+y^4-30x^2z^2+30y^2z^2+z^4)$
\item{}$(0, 3, 4, 6), -(16/15)x^4-(68/15)x^3y-(8/5)x^2y^2-(68/15)xy^3-(16/15)y^4-12xyz^2+z^4)$
\item{}$(0, 3, 4, 7), x^4+15y^4-17x^3z+45xy^2z+24x^2z^2-17xz^3+z^4)$
\item{}$(0, 3, 5, 6), x^4+15y^4+17x^3z-45xy^2z+24x^2z^2+17xz^3+z^4)$
\item{}$(0, 3, 5, 7), -(16/15)x^4+(68/15)x^3y-(8/5)x^2y^2+(68/15)xy^3-(16/15)y^4+12xyz^2+z^4)$
\item{}$(0, 3, 6, 7), (26/9)x^4-(10/3)x^2y^2+y^4-(10/3)x^2z^2+(10/3)y^2z^2+z^4)$
\item{}$(0, 4, 5, 6), x^4-(164/9)x^3y+6x^2y^2-(164/9)xy^3+y^4-(164/9)x^3z+12x^2yz-12xy^2z+(164/9)y^3z+6x^2z^2-12xyz^2+6y^2z^2-(164/9)xz^3+(164/9)yz^3+z^4)$
\item{}$(0, 4, 5, 7), (81/625)x^4+(492/125)x^3y+(54/25)x^2y^2-(84/125)xy^3+y^4-(492/125)x^3z-(108/25)x^2yz-(36/5)xy^2z+(28/25)y^3z+(54/25)x^2z^2+(36/5)xyz^2+6y^2z^2+(84/125)xz^3+(28/25)yz^3+z^4)$
\item{}$(0, 4, 6, 7), x^4+(84/125)x^3y+(54/25)x^2y^2-(492/125)xy^3+(81/625)y^4-(28/25)x^3z-(36/5)x^2yz+(108/25)xy^2z-(492/125)y^3z+6x^2z^2-(36/5)xyz^2+(54/25)y^2z^2-(28/25)xz^3+(84/125)yz^3+z^4)$
\item{}$(0, 5, 6, 7), (625/81)x^4-(700/81)x^3y+(1250/27)x^2y^2-(700/81)xy^3+(625/81)y^4-(140/27)x^3z+(500/9)x^2yz+(500/9)xy^2z-(140/27)y^3z+(50/3)x^2z^2+(100/3)xyz^2+(50/3)y^2z^2+(820/27)xz^3+(820/27)yz^3+z^4)$
\end{enumerate}

The automorphism group of $\det A=25(x^4+y^4+z^4)-34(x^2y^2+x^2z^2+y^2z^2)$  is the octahedral group $S_4$ of order $24$.
The $36$ preimages can be grouped into the  $S_4$-orbits
which are described in the following table.
$$\begin{array}{l|l|l}
(0,i,j,k)&\textrm{topological classification}&\textrm{Automorphism group}\\
\hline\\
\hline\\
\textrm{mother case in \ }(\ref{eq:4ovals})& \textrm{empty}&S_4\textrm{\ of order\ }24\\
\hline\\
(0,1,2,3) &\textrm{empty}& S_4\textrm{\ of order\ }24\\
\hline\\
\begin{array}{c}(0,1,2,4) (0,1,3,7) (0,2,3,6) (0,4,6,7)\\ 
(0,1,2,5) (0,1,3,6) (0,2,3,7) (0,5,6,7)\\ 
(0,1,2,7) (0,1,3,4) (0,2,3,5) (0,4,5,7)\end{array} &  \textrm{one oval}&\Z_2\textrm{\ of order\ }2\\
\hline\\
\begin{array}{c}(0,1,4,5) (0,1,6,7) (0,2,4,7)\\
 (0,2,5,6) (0,3,4,6) (0,3,5,7)\end{array} &\textrm{ one oval}&\Z_2\oplus\Z_2\textrm{\ of order\ }4\\
\hline\\
\begin{array}{c}(0,1,4,6) (0,1,5,7) (0,2,4,5)\\
 (0,2,6,7) (0,3,4,7) (0,3,5,6)\end{array} &\textrm{ one oval}&\Z_2\oplus\Z_2\textrm{\ of order\ }4\\
\hline\\
(0,1,2,6) (0,1,3,5) (0,2,3,4) (0,4,5,6) &  \textrm{one oval}&S_3\textrm{\ of order\ }6\\
\hline\\
(0,1,4,7)(0,2,5,7) (0,3,6,7) &\textrm{ empty}&D_4\textrm{\ of order\ }8\\
\hline\\
(0,1,5,6) (0,2,4,6) (0,3,4,5) &\textrm{ one oval}&D_4\textrm{\ of order\ }8\\
\end{array}$$
We print the figure corresponding to the star corresponding to $(0,1,5,6)$, provided by Emanuele Ventura.
\begin{figure}
\includegraphics[width=60mm]{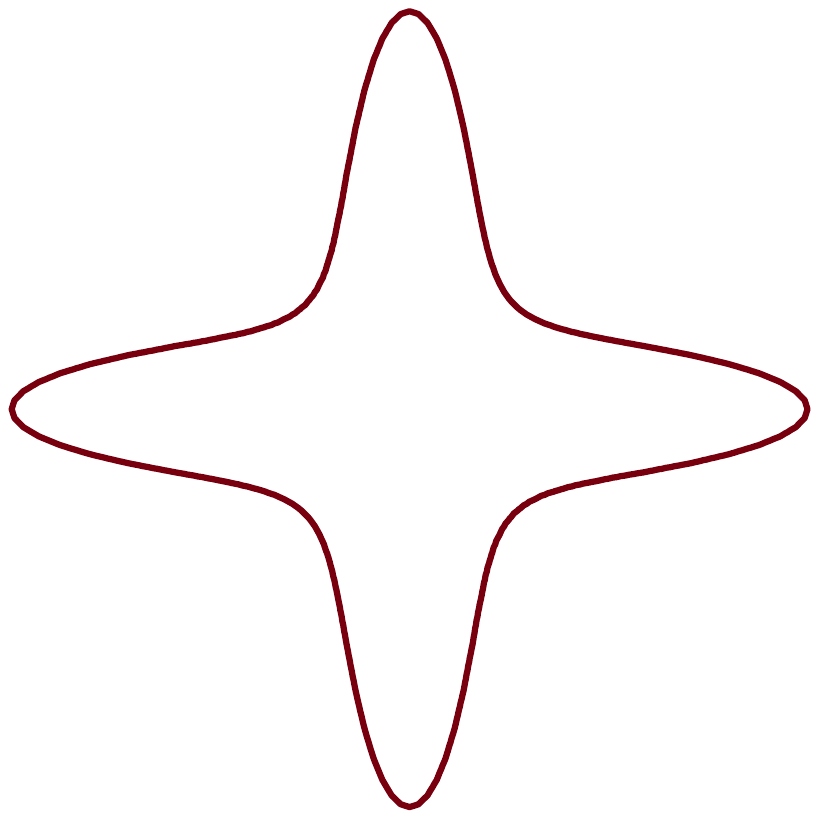}
\caption{One of the $36$ Scorza preimages of the Edge quartic}
\end{figure}

In the case of the Klein quartic we can follow an alternative path.
Note that the 28 bitangents of the Klein quartic are printed in \cite{Sh}.
Only four of them are real. The Klein quartic $f$ has the feature that 
\begin{equation}\label{eq:scf=f}f=S(f).\end{equation}
Ciani proves in \cite{Cia00} that the only quartics $f$ satisfying (\ref{eq:scf=f}) are the Klein quartic and the double conics (see also \cite{D})\footnote{It should be interesting to compute the weight decomposition of the tangent space at the fixed points of $S$, in order to understand the dynamics of the iteration of the Scorza map.}.
The other $35$ preimages of Klein quartic  through the Scorza map 
(besides the Klein quartic itself) were found
by Ciani in \cite{Cia01}, although Ciani describes them intrinsically and he does not provide the equations.
They divide in three $SL(3)$-classes, two classes (conjugate to each other) containing seven quartics each one, having $Aut$
the octahedral group of order $24$, and a (self-conjugate) class containing $21$ quartics, having $Aut=D_4$ the dihedral group of order $8$. 
This class contains three real quartics by (\ref{eq:tabletheta}).
 We have found the following solution (among the possible $35$), having octahedral symmetry. Let $\alpha=\tau+\tau^2+\tau^4=\frac{-1+\sqrt{-7}}{2}$.
If $f=x^4+y^4+z^4+2\alpha(x^3y+y^3z+z^3x)-\alpha(x^2y^2+y^2z^2+x^2z^2)+(\alpha-\frac{2}{3})(xy^3+yz^3+zx^3)-4(x^2yz+xy^2z+xyz^2)$
then $S(f)=xy^3+yz^3+zx^3$ is the Klein quartic.
The Klein quartic belongs to the pencil generated by the double (invariant) conic 
$\left[x^2+y^2+z^2+\alpha(xy+xz+yz)\right]^2$
and by the quadrilateral (given by four bitangents)
$(x+y+z)(x+(\tau^5+\tau^4+2\tau)y+(-2\tau^5-\tau^4+\tau^3-\tau^2-2*\tau)z)((-2\tau^5-\tau^4+\tau^3-\tau^2-2\tau)x+y+(\tau^5+\tau^4+2\tau)z)((\tau^5+\tau^4+2\tau)x+(-2\tau^5-\tau^4+\tau^3-\tau^2-2\tau)y+z)$. 

The M2 script is the following

\begin{verbatim}
KK=QQ
R=KK[x,y,z,u_0..u_3]

A = matrix({{0, x + 2*y, z + 2*x, -2*z + y}, 
  {x + 2*y, 0, 2*z + y, z - 2*x}, 
	  {z + 2*x, 2*z + y, 0, x - 2*y}, 
	  {-2*z + y, z - 2*x, x - 2*y, 0}});
----output will be x^4+30*x^2*y^2+y^4+30*x^2*z^2+30*y^2*z^2+z^4


det(A)	

----modification M'_{1234}, according to [Hesse]
A=matrix{{0,x-2*y,z-2*x,2*z+y},
    {x-2*y,0,-2*z+y,z+2*x},
    {z-2*x,-2*z+y,0,x+2*y},
    {2*z+y,z+2*x,x+2*y,0}}
-----output will be
--  x^4+(10/3)*x^2*y^2+y^4+(10/3)*x^2*z^2+(10/3)*y^2*z^2+z^4

	  
uu=transpose matrix {{u_0..u_3}}
--following M gives the sextic model in P^3 with coordinates u
M=diff(x,A)*uu|diff(y,A)*uu|diff(z,A)*uu

R2=KK[a_0..a_2,b_0..b_2,u_0..u_3,v_0..v_3,MonomialOrder=>{3,3,4,4}]
aa=matrix{{a_0..a_2}}, bb=matrix{{b_0..b_2}}
cs2=sub(M,R2)
----following gives the (3,3) correspondence on the sextic model with coord.(u,v)
cor33=minors(3,cs2)+sub(minors(3,cs2),apply(4,i->(u_i=>v_i)))+
ideal(matrix{{v_0..v_3}}*cs2)+
ideal(matrix{{u_0..u_3}}*sub(cs2,apply(4,i->(u_i=>v_i))))
----following lifts the (3,3) corresp. to the quartic model with coord. (a,b)
----it is a long computation, running in about 15 minutes on a PC, 
time eli=eliminate({u_0,u_1,u_2,u_3,v_0,v_1,v_2,v_3},
saturate(ideal(cs2*transpose aa)+ideal(sub(cs2,apply(4,i->(u_i=>v_i)))*
transpose bb)+cor33,ideal(u_0..u_3)*ideal(v_0..v_3)))
eli6=submatrix(gens eli,,{1..6})
-----the (3,3) correspondence is given by 6 biquadratic equations in a,b in eli6

R4=KK[x,y,z,a_0..a_2,b_0..b_2,c_0..c_14]
---now we compare the correspondence found with the correspondence obtained 
---starting from a generic quartic fc 
---given by rank (P_aP_b(fc)) <= 1
fc=(matrix{{c_0..c_14}}*transpose symmetricPower(4,matrix{{x,y,z}}))_(0,0)
f1=a_0*diff(x,fc)+a_1*diff(y,fc)+a_2*diff(z,fc)
f2=b_0*diff(x,f1)+b_1*diff(y,f1)+b_2*diff(z,f1)
xx=matrix{{x,y,z}}
ab=minors(2,diff(transpose xx,diff(xx,f2)));
---previous ab is the correspondence obtained from fc
----next aab is our correspondence
aab=ideal sub(eli6,R4);

----now to compare the two correspondences we encode in matrices with 36 columns 
---given by all biquadratic monomials
ab2=symmetricPower(2,matrix{{a_0..a_2}})**symmetricPower(2,matrix{{b_0..b_2}})
maab=diff(transpose ab2, gens aab)
---following must be nonzero
det submatrix(maab,{0..5},)
det submatrix(maab,{0,1,4,7,10,14},)
mab=diff(transpose ab2, gens ab);
----we impose that each column from mab is in the column span of maab
----in principle we have to check all 7*7 minors, but this is too expensive
----so we try with all 7*7 minors containing the first six rows, this is enough 
---in the examples found (not for Klein quartic!)
--if at the end we get infinitely many solutions, one has to add more conditions 
---given by more minors

-----------------following in case of Klein quartic
I=ideal(apply(36,i->det submatrix(maab|mab_{0},{0,1,4,7,10,14,i},)));
I=I+ideal(apply(36,i->det submatrix(maab|mab_{1},{0,1,4,7,10,14,i},)));
I=I+ideal(apply(36,i->det submatrix(maab|mab_{2},{0,1,4,7,10,14,i},)));
I=I+ideal(apply(36,i->det submatrix(maab|mab_{3},{0,1,4,7,10,14,i},)));
I=I+ideal(apply(36,i->det submatrix(maab|mab_{4},{0,1,4,7,10,14,i},)));
I=I+ideal(apply(36,i->det submatrix(maab|mab_{5},{0,1,4,7,10,14,i},)));
I=I+ideal(apply(36,i->det submatrix(maab|mab_{6},{0,1,4,7,10,14,i},)));
I=I+ideal(apply(36,i->det submatrix(maab|mab_{7},{0,1,4,7,10,14,i},)));
I=I+ideal(apply(36,i->det submatrix(maab|mab_{8},{0,1,4,7,10,14,i},)));
betti I
----------------------until here for Klein quartic

I=ideal(apply(30,i->det submatrix(maab|mab_{0},{0,1,2,3,4,5,i+6},)));
I=I+ideal(apply(30,i->det submatrix(maab|mab_{1},{0,1,2,3,4,5,i+6},)));
I=I+ideal(apply(30,i->det submatrix(maab|mab_{2},{0,1,2,3,4,5,i+6},)));
I=I+ideal(apply(30,i->det submatrix(maab|mab_{3},{0,1,2,3,4,5,i+6},)));
I=I+ideal(apply(30,i->det submatrix(maab|mab_{4},{0,1,2,3,4,5,i+6},)));
I=I+ideal(apply(30,i->det submatrix(maab|mab_{5},{0,1,2,3,4,5,i+6},)));
I=I+ideal(apply(30,i->det submatrix(maab|mab_{6},{0,1,2,3,4,5,i+6},)));
I=I+ideal(apply(30,i->det submatrix(maab|mab_{7},{0,1,2,3,4,5,i+6},)));
I=I+ideal(apply(30,i->det submatrix(maab|mab_{8},{0,1,2,3,4,5,i+6},)));
betti I
sI=saturate I;
betti sI
----sI contains now the base locus of Scorza map plus one single point,
----we have to saturate with base locus of Scorza, called bscorza

m=matrix{{0,z,-y},{-z,0,x},{y,-x,0}}
m9= diff(m,diff(xx,diff(transpose xx,fc)));
scorza=(mingens pfaffians(8,m9))_(0,0);
bscorza=ideal diff(symmetricPower(4,xx),scorza);

fiber=saturate(sI,bscorza)
----the following gives the quartic in the output
(sub(matrix {apply(15,i->(c_i)%fiber)},c_14=>1)*transpose symmetricPower(4,matrix{{x,y,z}}))_(0,0)

\end{verbatim}
\begin{alg}\ 

\begin{itemize}
\item{} INPUT: a plane quartic $f$
\item{} OUTPUT: the order of the automorphism group of linear invertible transformations which leave $f$ invariant
\end{itemize}
\end{alg}

This algorithm is straightforward. We add here the M2 script for the convenience of the reader.
The structure of groups of small size can be found by counting the elements of a given order.

\begin{verbatim}
---work on a finite field to run quicker
R2=ZZ/31991[g_0..g_8,y_0..y_2]
gg=transpose genericMatrix(R2,3,3)
-----we impose conditions to a unknown projective transformation gg
y4=symmetricPower(4,matrix{{y_0..y_2}})
y1=transpose matrix{{y_0..y_2}}
f=y_0^4+30*y_0^2*y_1^2+y_1^4+30*y_0^2*y_2^2+30*y_1^2*y_2^2+y_2^4 ---order 24 
IG=minors(2,diff(y4,sub(sub(f,R2),apply(3,i->(y_i=>((gg)*y1)_(i,0)))))||
diff(y4,f));
time sig=saturate(IG,ideal(det(gg)));
codim sig, degree sig
---if codim sig=8, then degree sig is the order of automorphism group
\end{verbatim}

 \noindent
  \textsc{g. ottaviani} -
  Dipartimento di Matematica e Informatica ``U. Dini'', Universit\`a di Firenze, viale Morgagni 67/A, 50134 Firenze (Italy). e-mail: \texttt{giorgio.ottaviani@unifi.it}
  
  \separation
  \noindent

\end{document}